\documentclass[11pt]{imsart}
\RequirePackage{amsthm,amsmath,amssymb,bm}
\RequirePackage[authoryear]{natbib}
\RequirePackage[colorlinks,citecolor=blue,urlcolor=blue]{hyperref}
\usepackage[left=3cm, right=3cm, top=2.5cm, bottom=2.5cm]{geometry}
\usepackage{xcolor}
\usepackage{epsfig}
\usepackage{graphicx}
\usepackage{multirow}
\usepackage{listings}
\usepackage{float}
\usepackage{extarrows}
\numberwithin{equation}{section}
\usepackage{booktabs}
\usepackage{threeparttable}
\usepackage{tabularx}
\usepackage{makecell}

\theoremstyle{plain}
\newtheorem{thm}{Theorem}[section]
\newtheorem{lemma}[thm]{Lemma}

\newtheorem{proposition}[thm]{Proposition}

\theoremstyle{definition}
\newtheorem{definition}{Definition}[section]
\newtheorem{condition}{Condition}[section]

\theoremstyle{remark}
\newtheorem{remark}{Remark}[section]

\newcommand{\bS}{{\bf S}}

\newcommand{\bR}{{\bf R}}
\newcommand{\br}{{\bf r}}
\newcommand{\bF}{{\bf F}}
\newcommand{\bH}{{\bf H}}

\newcommand{\bI}{{\bf I}}
\newcommand{\bX}{{\bf X}}
\newcommand{\bY}{{\bf Y}}

\newcommand{\bw}{{\bf w}}

\newcommand{\bx}{{\bf x}}
\newcommand{\by}{{\bf y}}

\newcommand{\bD}{{\bf D}}
\newcommand{\bA}{{\bf A}}

\newcommand{\bB}{{\bf B}}

\newcommand{\va}{{\varepsilon}}
\newcommand{\bSig}{{\boldsymbol\Sigma}}

\newcommand{\bga}{{\mathbf G}}
\newcommand{\mb}{\mathbf}
\renewcommand{\(}{\left(}
\renewcommand{\)}{\right)}
\newcommand{\re}{{\rm E}}
\newcommand{\E}{{\rm E}}
\newcommand{\tr}{{\rm tr}}
\newcommand{\diag}{{\rm diag}}
\newcommand{\R}{{\mathbb R}}
\newcommand{\C}{{\mathbb C}}
\newcommand{\Cov}{{\rm Cov}}
\newcommand{\Var}{{\rm Var}}

\allowdisplaybreaks

\begin{document}

\begin{frontmatter}
\title{The Geometry of Spectral Fluctuations: On Near-Optimal Conditions for Universal Gaussian CLTs, with Statistical Applications}
\runtitle{The Geometry of Spectral Fluctuations}

\begin{aug}
\author[A]{\fnms{Yanqing} \snm{Yin}
			\ead[label=e1]{yinyq799@nenu.edu.cn}}
			\and
\author[B]{\fnms{Wang} \snm{Zhou}
			\ead[label=e2]{wangzhou@nus.edu.sg}}

\runauthor{Y. Yin and W. Zhou}

\address[A]{School of Statistics and Data Science,
			Nanjing Audit University\printead[presep={,\ }]{e1}}

\address[B]{Department of Statistics and Data Science, National University of Singapore \printead[presep={,\ }]{e2}}
\end{aug}

\begin{abstract}
We study linear spectral statistics of high dimensional sample covariance
matrices in a regime where the empirical spectral distribution remains governed by the classical sample covariance law  but the fluctuation theory is nonclassical. Our
starting point is a decomposition of the covariance of centered quadratic forms
into a universal Gaussian part and a model dependent fourth order correction.
This leads to an abstract framework, termed GHOST, for universal Gaussian
central limit theorems under structured fourth order effects. Under this
framework, we prove a Gaussian central limit theorem for linear spectral
statistics, with explicit mean and covariance corrections determined by a
bilinear fourth order kernel. Boundary examples show that the conditions are close to necessary for a broad universal Gaussian closure.

We then develop a blockwise mixed radial model that verifies the abstract
assumptions and makes the correction explicit. The correction splits into an
entrywise fourth moment component and a blockwise energy fluctuation component.
The latter may change the fluctuation scale, leading to a phase transition at the level of fluctuations.

As an application, we study sphericity testing. Under the spherical null, the
general correction collapses to a single scalar parameter, yielding a feasible
data driven correction of John's test.
\end{abstract}



\end{frontmatter}

\section{Introduction}

Linear spectral statistics (LSS) of sample covariance matrices are central objects in high dimensional statistics and random matrix theory. They arise naturally in global inference problems for covariance structure, including hypothesis testing, goodness of fit assessment, and the calibration of eigenvalue based procedures. For sample covariance matrices generated from independent component models, the asymptotic theory of LSS is by now classical, beginning with the central limit theorem of \citet{Bai2004a}; see also \citet{ZhengB15S}. A substantial literature has extended this theory to other structured settings, including elliptical populations, mixtures, and more general dependence structures; see, for example, \citet{Hu2019,LiY18S,BaoP15A,bao2024spectral,li2024spectral,BaoL15S,Yin,zhangzheng2022,liu2023clt,Qiu2023,YangP15I,GaoH17H,Li2021,hu2026generalized}.

A recurring theme in this literature is that first order and second order spectral behavior need not coincide. A model may still exhibit a classical Mar\v{c}enko--Pastur (M--P) limit for the empirical spectral distribution, while the fluctuation of LSS is no longer correctly calibrated by the classical Gaussian theory. This distinction is statistically important: many high dimensional procedures are built from global spectral functionals whose deterministic limits are robust across broad model classes, whereas their inferential validity depends on the fluctuation theory.

Much of the post-Bai--Silverstein literature may be viewed from this perspective. For many specific models, the main issue is no longer whether an LSS central limit theorem exists, but how the non-Gaussian or heterogeneous structure modifies the classical fluctuation formula. Existing works have identified such corrections in several different settings, including non-Gaussian independent component models through fourth moments, elliptical models through radial dependence, mixture models through latent scale heterogeneity, and more recent models involving random missingness. These results have greatly expanded the scope of LSS theory, but they are still largely model specific. From the viewpoint of random matrix theory and high dimensional inference, it is therefore natural to ask whether these apparently different corrections are manifestations of a deeper common mechanism.

The present paper is motivated precisely by this question. We formulate an abstract framework based on the covariance structure of centered quadratic forms. For deterministic matrices $\bA_n$ and $\bB_n$, consider
\[
\br_1^\top \bA_n\br_1-\tr(\bA_n\bSig_n),
\qquad
\br_1^\top \bB_n\br_1-\tr(\bB_n\bSig_n).
\]
Their covariance naturally splits into two parts: a universal Gaussian contribution, which depends only on the covariance matrix $\bSig_n$, and a model dependent fourth order correction. The latter is encoded by a bilinear functional $\Gamma_n(\cdot,\cdot)$, which may be viewed as the effective action of the fourth order cumulant structure on matrix directions relevant to the LSS problem. In this sense, $\Gamma_n$ is the object through which non-Gaussianity and heterogeneity enter the fluctuation theory. We formulate an abstract structural framework for such corrections, referred to as GHOST (\emph{Generalized Higher Order Spectral Template}). Its purpose is to isolate the conditions under which fourth order effects modify the LSS fluctuation formulas while preserving Gaussian universality. The point is not the mere existence of such a correction, but whether it remains weak enough, regular enough, and stable enough to preserve Gaussian universality at second order while leaving the first order spectral law unchanged.

The value of such a framework is both conceptual and practical. Conceptually,
it identifies the structural source of the nonclassical fluctuation correction
within an otherwise classical sample covariance regime. Practically, this is
the regime most relevant for statistical calibration. In many applications,
the correction is inherently model dependent and must ultimately be estimated
from data. A structural theory is therefore more useful than a collection of
model specific formulas, because it identifies the common object driving the
correction and clarifies when Gaussian universality can still survive.

Our first contribution is to show that, under the GHOST conditions, linear
spectral statistics remain asymptotically Gaussian with explicit mean and
covariance corrections driven by a fourth order kernel. Boundary examples show
that these conditions are close to necessary for universal Gaussian fluctuation theory.

Our second contribution is to show, through a concrete heterogeneous model,
that the effect of fourth order structure is not limited to changing the
limiting mean and covariance. We study a blockwise mixed radial construction in
which the correction kernel becomes explicit and splits into two qualitatively
different parts: an entrywise fourth moment component of Hadamard trace type,
and a blockwise energy fluctuation component of trace product type. This model reveals a genuine phase transition in the fluctuation theory. While the first
order spectral law remains classical, the fluctuation may already leave the
classical $n$ scale regime, and the normalization can change from $n$ to
$\sqrt{n}$ according to the strength of the blockwise energy fluctuation.

Our third contribution is a statistical application to sphericity testing. Under the spherical null, the general fourth order correction collapses to a single scalar parameter measuring excess sample energy fluctuation beyond the Gaussian benchmark. This yields a feasible corrected John's test based on a fully data driven estimator of that scalar parameter, without requiring estimation of a high dimensional fourth order tensor or other comparably complex nuisance objects.

It is useful to position the present work relative to several nearby fluctuation regimes. In the classical independent component setting, the Bai--Silverstein theory yields an $n$-scale central limit theorem for LSS within the M--P class. For light tailed elliptical populations and random missing data model, the empirical spectral distribution also remains of classical sample covariance type and the fluctuation is still on the classical scale, although the limiting mean and covariance are modified by radial dependence; see \citet{Hu2019,li2024spectral}. By contrast, in mixture models the first order spectral law may already leave the M--P class, and the fluctuation scale may drop to $\sqrt n$; see \citet{LiY18S}. What distinguishes the present regime is that second order spectral instability can emerge strictly before first order spectral approximation breaks down.

The rest of the paper is organized as follows. Section~\ref{sec:framework}
introduces the general framework and assumptions. Section~\ref{sec:main_results}
states the main fluctuation theorem for linear spectral statistics.
Section~\ref{sec:block} develops the blockwise mixed radial model and its
phase transition behavior. Section~\ref{sec:sphericity} studies the corrected
John's test under the spherical null. Section~\ref{sec:numerics} reports
numerical results. To keep the main text focused on the structural ideas and
statistical consequences, technical proofs of the main results, together with
several auxiliary derivations, are deferred to the Supplementary Material \citet{Yin2026}.

\section{General model and assumptions}\label{sec:framework}

\subsection{Model description}

For any Hermitian matrix $\bA$ of size $p\times p$, let
\[
F^{\bA}(x)=\frac1p\sum_{j=1}^p {\mathbb I}(\lambda_j^{\bA}\le x)
\]
denote its empirical spectral distribution (ESD), where
$\lambda_1^{\bA},\dots,\lambda_p^{\bA}$ are the eigenvalues of $\bA$. For a probability distribution $F$ on $\R$, denote its Stieltjes transform by
\[
m_F(z)=\int\frac{1}{x-z}\,dF(x),\qquad z\in\C^+:=\{z\in\C:\Im z>0\}.
\]

We consider the sample covariance matrix
\[
\bS_n=\frac1n\sum_{k=1}^n\br_k\br_k^\top,
\]
where $\br_1,\dots,\br_n$ are independent $p$-dimensional observations. Our aim is to develop an LSS fluctuation theory under a structural framework formulated in terms of centered quadratic forms.

\subsection{The GHOST framework}

We now formulate the abstract structure underlying the fluctuation theory. The
starting point is not a particular parametric model, but the covariance
structure of centered quadratic forms. 

Recall the fourth order joint cumulant
\[
{\rm Cum}(\zeta_1,\zeta_2,\zeta_3,\zeta_4)
=
\E(\zeta_1\zeta_2\zeta_3\zeta_4)
-\E(\zeta_1\zeta_2)\E(\zeta_3\zeta_4)
-\E(\zeta_1\zeta_3)\E(\zeta_2\zeta_4)
-\E(\zeta_1\zeta_4)\E(\zeta_2\zeta_3).
\]

\begin{lemma}\label{lem:gamma_cumulant}
Let $\by_n$ satisfy $\E\by_n=\mathbf 0$, $\E(\by_n\by_n^\top)=\bSig_n$, and
$\E\|\by_n\|_2^4<\infty$. Then for any deterministic matrices $\bA_n,\bB_n$,
\begin{align}
\Cov\bigl(Q_n(\bA_n),Q_n(\bB_n)\bigr)
=
\tr(\bA_n\bSig_n\bB_n\bSig_n)
+\tr(\bA_n\bSig_n\bB_n^\top\bSig_n)
+\Gamma_n(\bA_n,\bB_n),
\label{eq:gamma_cumulant_decomp}
\end{align}
where
\[
Q_n(\bA_n)=\by_n^\top\bA_n\by_n-\tr(\bA_n\bSig_n),
\qquad
\Gamma_n(\bA_n,\bB_n)
=
\sum_{i,j,k,l=1}^p a_{ij}b_{kl}\kappa^{(n)}_{ijkl},
\]
and $\kappa^{(n)}_{ijkl}={\rm Cum}(y_i,y_j,y_k,y_l)$ is the fourth order cumulant tensor of $\by_n$.
\end{lemma}

In the symmetric case $\bA_n=\bA_n^\top$ and $\bB_n=\bB_n^\top$, the Gaussian part in \eqref{eq:gamma_cumulant_decomp} reduces to
\[
2\tr(\bA_n\bSig_n\bB_n\bSig_n).
\]
Motivated by \eqref{eq:gamma_cumulant_decomp}, we formulate the theory through an abstract bilinear correction functional $\Gamma_n(\cdot,\cdot)$.

\begin{definition}[GHOST class]
A sequence of random vectors is said to belong to the \emph{GHOST class} if it satisfies the basic high dimensional covariance assumptions, the regularity conditions on the quadratic form correction kernel, the quadratic form moment regularity condition, and the high probability spectral localization condition stated below.
\end{definition}

\subsubsection{Basic high dimensional covariance setting}

\begin{condition}\label{cond:g1}
The vectors $\br_1,\dots,\br_n$ are independent and identically distributed, and
\[
\E\br_1=\mathbf 0,
\qquad
\E(\br_1\br_1^\top)=\bSig_n.
\]
\end{condition}

\begin{condition}\label{cond:g2}
The deterministic matrices $\bSig_n$ are nonnegative definite, and the ESD
$H_n:=F^{\bSig_n}$ converges weakly to a nondegenerate probability distribution $H$.
Moreover, $\|\bSig_n\|$ is uniformly bounded in $n$.
\end{condition}

\begin{condition}\label{cond:g3}
The aspect ratio satisfies
\[
c_n=\frac{p}{n}\to c\in(0,\infty).
\]
\end{condition}

\subsubsection{Regularity of the fourth order correction kernel}

For a possibly random complex $p\times p$ matrix ${\mathbf M}_n$, define the bilinear
$L_1$-seminorm
\[
\|{\mathbf M}_n\|_{\mathrm{bil},L_1}
:=
\sup_{\|a_n\|_2=\|b_n\|_2=1}
\E\big|a_n^\top {\mathbf M}_n b_n\big|,
\]
where the supremum is taken over all deterministic unit vectors
$a_n,b_n\in\mathbb R^p$. Note that for deterministic ${\mathbf M}_n$, the bilinear
$L_1$-seminorm degenerates to the spectral norm.

\begin{condition}[$L_1$-stability of the fourth order correction kernel]
\label{cond:g4}
For any possibly random complex $p\times p$ matrices
$\bA_n,\widetilde{\bA}_n,\bB_n,\widetilde{\bB}_n$ satisfying
\[
\|\bA_n\|+\|\widetilde{\bA}_n\|+\|\bB_n\|+\|\widetilde{\bB}_n\|=O(1)
\]
uniformly, one has
\begin{align}
\Big\|
\Gamma_n(\bA_n,\bB_n)-\Gamma_n(\widetilde{\bA}_n,\bB_n)
\Big\|_{L_1}
&=
o\Big(
p^{3/2}\,\|\bA_n-\widetilde{\bA}_n\|_{\mathrm{bil},L_1}
\Big),
\label{eq:gamma_bilL1_cont1}\\
\Big\|
\Gamma_n(\bA_n,\bB_n)-\Gamma_n(\bA_n,\widetilde{\bB}_n)
\Big\|_{L_1}
&=
o\Big(
p^{3/2}\,\|\bB_n-\widetilde{\bB}_n\|_{\mathrm{bil},L_1}
\Big).
\label{eq:gamma_bilL1_cont2}
\end{align}
\end{condition}

Condition~\ref{cond:g4} should be viewed as a structural regularity condition on the \emph{effective} fourth order correction kernel rather than as a purely technical continuity assumption. Its direct role in the proof is to justify deterministic replacement: the resolvent type matrices entering the fluctuation analysis are replaced by deterministic equivalents whose errors are small in the bilinear $L_1$ sense, and Condition~\ref{cond:g4} guarantees that such perturbations do not induce a macroscopic change in the fourth order correction term. Without such stability, the fourth order contribution could remain too sensitive to weak replacement errors of the resolvent to be captured by a deterministic limit.

More fundamentally, Condition~\ref{cond:g4} controls the effective complexity of the fourth order cumulant tensor only along matrix directions relevant to the LSS fluctuation problem. Indeed,
\[
\Gamma_n(\bA_n,\bB_n)
=
\sum_{i,j,k,l=1}^p a_{ij}b_{kl}\kappa^{(n)}_{ijkl},
\]
so the condition does not constrain the tensor entrywise; rather, it rules out fourth order interactions whose action on weakly perturbed resolvent directions is too unstable to admit a tractable bulk correction formula. In the main structured regimes of interest, such low-complexity behavior is natural: in independent component models, the fourth order correction survives essentially on diagonal index patterns and reduces to a Hadamard type form, whereas in elliptical or radial models symmetry forces the correction to act through a few global energy channels, yielding trace product terms. The blockwise mixed radial model studied later combines precisely these two mechanisms.

%Under uniformly bounded coordinate fourth moments, the fourth order correction is generically at most of order $p^2$. Thus, the substantive issue behind Condition~\ref{cond:g4} is not the existence of a coarse growth bound, but whether the correction admits the stronger stability needed to remain tractable at the fluctuation scale relevant for the CLT.

\begin{remark}[A tensor level sufficient condition for Condition~\ref{cond:g4}]
\label{rem:g7_tensor}
A sufficient condition for Condition~\ref{cond:g4} is that
\begin{align}
\sum_{i,j=1}^p
\left(
\sum_{k,l=1}^p (\kappa^{(n)}_{ijkl})^2
\right)^{1/2}
&=o(p),
\label{eq:g7_tensor_suff1}\\
\sum_{k,l=1}^p
\left(
\sum_{i,j=1}^p (\kappa^{(n)}_{ijkl})^2
\right)^{1/2}
&=o(p).
\label{eq:g7_tensor_suff2}
\end{align}
Indeed, writing $\Delta\bA_n=\bA_n-\widetilde{\bA}_n$, we have
\begin{align*}
\big\|
\Gamma_n(\bA_n,\bB_n)-\Gamma_n(\widetilde{\bA}_n,\bB_n)
\big\|_{L_1}
&=
\big\|\Gamma_n(\Delta\bA_n,\bB_n)\big\|_{L_1}\\
&\le
\max_{i,j}\|\Delta\bA_{n,ij}\|_{L_1}
\sum_{i,j=1}^p
\left|
\sum_{k,l=1}^p b_{kl}\kappa^{(n)}_{ijkl}
\right|\\
&\le
\max_{i,j}\|\Delta\bA_{n,ij}\|_{L_1}\,\big\|\|\bB_n\|_F\big\|_{L_1}
\sum_{i,j=1}^p
\left(
\sum_{k,l=1}^p (\kappa^{(n)}_{ijkl})^2
\right)^{1/2}.
\end{align*}
Hence, whenever
\[
\|\bB_n\|_F=O(p^{1/2}),
\]
one obtains
\[
\Big\|
\Gamma_n(\bA_n,\bB_n)-\Gamma_n(\widetilde{\bA}_n,\bB_n)
\Big\|_{L_1}=
o\Big(
p^{3/2}\,\max_{i,j}\|\Delta\bA_{n,ij}\|_{L_1}
\Big),
\]
Note that $\max_{i,j}\|\Delta\bA_{n,ij}\|_{L_1}$ is bounded by $\|\bA_n-\widetilde{\bA}_n\|_{\mathrm{bil},L_1}$. The proof for the second argument is identical.
\end{remark}

\begin{remark}[Statistical meaning of Condition~\ref{cond:g4}]
\label{rem:g4_sqrtn}
Condition~\ref{cond:g4} can be interpreted as a weak replacement stability requirement for the fourth order correction. In the LSS fluctuation formulas below, the non-Gaussian contribution always appears through the scaled quantity
\[
\frac1n\Gamma_n(\cdot,\cdot).
\]
Condition~\ref{cond:g4} states that matrix replacements which are negligible at the natural bilinear $L_1$ scale $p^{-1/2}$ change the fourth order correction only by an $o(p)$ term, and hence do not affect the limiting correction after normalization by $n^{-1}$.

This point is statistically important. Any feasible implementation of the fourth order correction must replace the deterministic matrix arguments in $\Gamma_n(\cdot,\cdot)$ by sample based approximations or plug in estimators. If Condition~\ref{cond:g4} fails substantially, then the correction may become too sensitive even to weak bilinear replacement errors, so that plug in uncertainty enters at the same order as the target fluctuation itself. Thus, even in situations where the first order M--P law may still survive, second order calibration can become unstable.
\end{remark}

%\begin{remark}
%Although tensor-level conditions on $\kappa^{(n)}_{ijkl}$ are mainly imposed at the model level, they may still be probed empirically through complexity diagnostics. For example, one may consider a sample analogue of
%\[
%\sum_{i,j=1}^p
%\left(
%\sum_{k,l=1}^p (\kappa^{(n)}_{ijkl})^2
%\right)^{1/2},
%\]
%or a subsampled/projection based version of it. Such quantities should be viewed only as diagnostics: if they are excessively large, this indicates that the effective fourth order interaction may be too rich for the present deterministic approximation to remain accurate.
%\end{remark}

\subsubsection{Quadratic form moment regularity and spectral localization}

\begin{condition}[Rosenthal type quadratic form moment regularity]
\label{cond:g5}
For
\[
Q_n(\bA_n)=\br_1^\top \bA_n\br_1-\tr(\bA_n\bSig_n),
\]
assume that for every integer $q\ge2$, there exists a constant $C_q>0$ and a deterministic sequence $\Delta_n\downarrow0$
such that for every deterministic matrix $\bA_n$ with uniformly bounded spectral norm,
\begin{align}\label{eq:g5_final}
\E\big|Q_n(\bA_n)\big|^q
\le
C_q\,\Delta_n^{\,q-2}n^{\,q-2}\,
V_n^2(\bA_n)\,
\|\bSig_n^{1/2}\bA_n\bSig_n^{1/2}\|^{\,q-2},
\end{align}
where
\[
V_n^2(\bA_n)
=
\tr\!\Big((\bA_n\bSig_n)^2\Big)
+
\tr\!\Big(\bA_n\bSig_n\bA_n^\top\bSig_n\Big)
+
\Gamma_n(\bA_n,\bA_n).
\]
\end{condition}

\begin{condition}[high probability spectral confinement]
\label{cond:g6}
Let
\begin{align}
\mathcal I=
\left[
\liminf_n \lambda_{\min}^{\bSig_n} I_{(0,1)}(c)(1-\sqrt c)^2,\,
\limsup_n \lambda_{\max}^{\bSig_n}(1+\sqrt c)^2
\right].
\label{eq:confine_interval}
\end{align}
For any fixed $\delta>0$, define
\[
\mathcal I_\delta=\{x\in\R:\operatorname{dist}(x,\mathcal I)\le\delta\}.
\]
Then, for every $L>0$, there exists $n_0=n_0(\delta,L)$ such that for all $n\ge n_0$,
\[
\Pr\big(\operatorname{spec}(\bS_n)\subset \mathcal I_\delta\big)\ge 1-n^{-L}.
\]
\end{condition}

\begin{remark}[Interpretation of Conditions~\ref{cond:g4}--\ref{cond:g6}]
Conditions~\ref{cond:g4}--\ref{cond:g6} together delineate a regime in which bulk spectral inference remains both mathematically tractable and statistically stable. Condition~\ref{cond:g4} controls the stability of the effective fourth order correction kernel under weak matrix replacements, Condition~\ref{cond:g5} controls directional quadratic-form fluctuations at the natural variance scale and excludes rare but macroscopically influential spikes, and Condition~\ref{cond:g6} ensures that linear spectral statistics remain a genuine bulk problem by excluding escaping eigenvalues. The boundary examples in Section~\ref{subsec:ghost_conditions_discussion} show that these three conditions play genuinely different roles in the universal Gaussian fluctuation theory.
\end{remark}

Many models previously studied through separate LSS central limit theorems can be viewed from a common GHOST perspective, with the effective fourth order correction falling into a small number of structural types.

\begin{table}[!htbp]
\centering
\footnotesize
\setlength{\tabcolsep}{5pt}
\renewcommand{\arraystretch}{1.12}
\begin{threeparttable}
\caption{Representative model classes related to the GHOST viewpoint}
\label{tab:ghost_models}
\begin{tabular}{@{}llll@{}}
\toprule
Label & Effective vector & $\bSig$ & Dominant fourth order structure \\
\midrule

IC-SCM
&
\makecell[l]{$\br=\bSig_0^{1/2}\bf z$}
&
\makecell[l]{$\bSig_0$}
&
\makecell[l]{Diagonal sector fourth order correction}
\\[4pt]

IC-Corr\tnote{a}
&
\makecell[l]{$\tilde{\br}_i=\sqrt{n}\br_i/\|\br_i\|_2$}
&
\makecell[l]{$n\E\!\left(\br_i\br_i^\top/\|\br_i\|_2^2\right)$}
&
\makecell[l]{Row normalization; after $p$--$n$ exchange,\\
independent row Gram structure}
\\[4pt]

E-SSCM\tnote{b}
&
\makecell[l]{$\tilde{\br}=\sqrt{p}\br/\|\br\|_2$}
&
\makecell[l]{$p\Cov(\br/\|\br\|_2)$}
&
\makecell[l]{GHOST applied after transformation;\\
includes heavy tailed elliptical inputs}
\\[4pt]

BMR
&
\makecell[l]{$\br=\bga(\rho_x\bx^\top,\rho_y\by^\top)^\top$}
&
\makecell[l]{$\bga\bga^\top$}
&
\makecell[l]{Hadamard type term plus\\
block energy trace product term}
\\[4pt]

MCAR\tnote{c}
&
\makecell[l]{$\tilde{\br}=\bB\circ\br$}
&
\makecell[l]{${\mathbf P\bSig_0\mathbf P+(\mathbf P^2-\mathbf P)\circ \bSig_0 }$}
&
\makecell[l]{Correction driven jointly by\\
signal non-Gaussianity and masking}
\\[4pt]

PSG
&
\makecell[l]{$\br$}
&
\makecell[l]{$\Cov(\br)$}
&
\makecell[l]{Correction driven by covariance structure}
\\

\bottomrule
\end{tabular}

\begin{tablenotes}[flushleft]
\footnotesize
\item IC-SCM: independent component sample covariance matrix under finite fourth moment condition.
\item IC-Corr: sample correlation matrix under independent components.
\item E-SSCM: spatial sign covariance matrix under elliptical distributions. 
\item BMR: blockwise mixed radial model discussed in this paper.
\item MCAR: completely random missingness model under finite fourth moment condition.
\item PSG: proper sub-Gaussian vector model satisfies condition ~\ref{cond:g4}.
\item[a] For the sample correlation matrix under independent components, the normalization is rowwise. After interchanging the roles of dimension and sample size, the associated Gram matrix fits the GHOST viewpoint.
\item[b] Here GHOST is applied to the transformed spatial sign vector rather than to the raw observation. Here $\br$ is multivariate normal distributed. 
\item[c] The exact form of $\bSig$ depends on the masking scheme. Here $\bSig_0=\Cov(\br)$ and $\mathbf P$ is the expectation matrix of the diagonal matrix $\mathbf B$ with independent Bernoulli entries.
\end{tablenotes}
\end{threeparttable}
\end{table}

\subsubsection{Verification of Conditions~\ref{cond:g5}--\ref{cond:g6} in two standard regimes}
\paragraph{Sub-Gaussian regime}
Suppose that the whitened vector
\[
\bx_n=\bSig_n^{-1/2}\br_1
\]
is isotropic and uniformly sub-Gaussian in the sense that
\[
\sup_n\sup_{\|u\|_2=1}\|\langle u,\bx_n\rangle\|_{\psi_2}\le K
\]
for some constant $K>0$. Then standard Hanson--Wright inequalities imply that
for every fixed integer $q\ge2$, there exists a constant $C_q>0$ such that for
every deterministic matrix $\bA_n$ with uniformly bounded spectral norm,
\[
\big\|
\br_1^\top \bA_n \br_1-\tr(\bA_n\bSig_n)
\big\|_{L^q}
\le
C_q\Big(
\|\bSig_n^{1/2}\bA_n\bSig_n^{1/2}\|_F
+
\|\bSig_n^{1/2}\bA_n\bSig_n^{1/2}\|
\Big).
\]
Hence Condition~\ref{cond:g5} holds a fortiori in this regime.

For the proof of the LSS central limit theorem based on contour integration, resolvent expansion, and martingale decomposition,
it is already sufficient to know that the eigenvalues of $\bS_n$ are contained,
with overwhelming probability, in some fixed compact interval.  A convenient route is the
$\varepsilon$-net method for operator norms, as developed in \cite{vershynin2018high}: one first controls bilinear forms
over a fixed pair of directions, then passes to the full operator norm through
a finite net of the unit sphere. This yields high probability bounds for
$
\|\bS_n-\bSig_n\|,
$
and consequently high probability control of the extreme eigenvalues of
$\bS_n$. Once such a high probability
spectral bound is available, one may enlarge the contour so that it still
encloses the limiting support and, with overwhelming probability, also encloses
the full empirical spectrum. This is enough for the contour proof, and the
resulting residue calculation is unchanged.

If one wishes to sharpen this coarse confinement to a localization statement relative to the limiting support, then one may further adjust the classical route in \cite{BaiS98N} for excluding eigenvalues outside the limiting support with overwhelming probability.

\paragraph{Independent component regime and truncation reduction}
In the independent component setting, Conditions~\ref{cond:g5} and
\ref{cond:g6} are typically verified not for the raw coordinates directly, but
for their truncated, recentered, and renormalized counterparts. The key point
is that this reduction does not affect the LSS central limit theorem.

More precisely, suppose
\[
\br_1=\bSig_n^{1/2}{\mathbf z}_1,
\qquad
{\mathbf z}_1=(z_{11},\dots,z_{1p})^\top,
\]
where the coordinates of ${\mathbf z}_1$ are independent, centered, have unit
variance, and finite fourth moment. Following the standard truncation scheme,
one first replaces ${\mathbf z}_1$ by its truncated and recentered version,
still denoted by ${\mathbf z}_1$, such that
\[
|z_{1i}|\le \eta_n\sqrt n
\qquad\text{for all }1\le i\le p,
\]
with $\eta_n\downarrow0$ sufficiently slowly.

For the truncated model, fixed order Rosenthal type bounds for quadratic forms,
see (9.9.3) in \cite{BaiS10S}, yield the required moment control. The role of
the factor $\Delta_n$ in Condition~\ref{cond:g5} is precisely to encode the
smallness created by truncation: under the bound $|z_{1i}|\le \eta_n\sqrt n$,
higher moments of quadratic form increments acquire an additional factor of
order $\eta_n^2$, which is exactly the type of vanishing input needed to rule
out rare but macroscopically influential spikes. The sparse spike
counterexample in Section~\ref{subsubsec:g5_delta_necessary} shows that such a
vanishing factor cannot, in general, be removed.

Condition~\ref{cond:g6} for the truncated model follows from classical spectral
edge theory, as shown in \cite{yin1988limit} or p.~263 of \cite{BaiS10S}.

\subsection{Boundary examples and the near-optimality of the GHOST conditions}
\label{subsec:ghost_conditions_discussion}

The discussion below suggests that the GHOST conditions are close to a structural characterization of the universal Gaussian regime. Condition~\ref{cond:g6} captures first order spectral confinement, Condition~\ref{cond:g5} excludes rare but macroscopically influential quadratic fluctuations, and Condition~\ref{cond:g4} governs whether the fourth order correction admits a universal deterministic closure.

\subsubsection{Sample correlation matrices as a boundary example for Condition~\ref{cond:g4}}
\label{subsubsec:corr_g4_boundary}

Sample correlation matrices provide a natural boundary example for the GHOST framework. Let $\bX=(x_{jk})_{1\le j\le p,\,1\le k\le n}$ be a $p\times n$ data matrix, and write
\[
\widehat{\bR}=\bX^\top (\diag(\bS))^{-1}\bX,\qquad \bS=\bX\bX^\top.
\]
Equivalently, if $\bY=(Y_1,\dots,Y_p)$ with $Y_j=x_j/\|x_j\|\in\mathbb{R}^n$, where $x_j$ denotes the $j$th row of $\bX$, then
\[
\widehat{\bR}=\sum_{j=1}^p Y_jY_j^\top.
\]
Thus $\widehat{\bR}$ is a Gram-type matrix generated by self-normalized random vectors. After rescaling $r_j=\sqrt{n}\,Y_j$, one has $\|r_j\|^2=n$. Hence sample correlation matrices lie exactly on the boundary between classical covariance-type models and normalization-driven models: the first order spectral behavior remains of classical sample covariance type, while the second order fluctuation mechanism is altered by self-normalization.

This already singles out a frozen direction. For
\[
Q_n(\bA)=r^\top \bA r-\tr(\bA\bSig_n),\qquad \bSig_n=\E(rr^\top),
\]
one has $r^\top {\bI}_n r=n$, hence $Q_n(\bI_n)=0$. Thus self-normalization removes the energy fluctuation in the identity direction exactly. However, for general linear spectral statistics, the relevant matrix arguments are not only multiples of $I_n$, but the entire family of resolvent-type directions arising in the contour method. The real issue is therefore whether the fourth order correction kernel remains negligible and stable on nontrivial shape directions.

The first order confinement condition Condition~\ref{cond:g6} is already known to hold in this model under heavy tails. Under the regularly varying tail assumption
\[
\Pr(|X|>x)\sim l(x)x^{-\alpha},\qquad \alpha\in(2,4],\qquad \frac{n}{p}\to \phi\in(0,\infty)\setminus\{1\},
\]
Li, Pan, Xie and Zhou proved that $\|\widehat{\bR}\|$ is bounded with high probability; see Lemma~2.11 of \cite{LiPanXieZhou2024}. Hence the analogue of Condition~\ref{cond:g6} already holds throughout the whole regime $\alpha>2$.

We now show that the genuinely critical condition is Condition~\ref{cond:g4}. By exchangeability and $\sum_i r_i^2=n$, the diagonal entries of $\bSig_n$ are all equal to $1$. Choose a deterministic trace-zero diagonal matrix (assume $n$ is even without loss of generality)
\[
\bA_n=\diag(a_1,\dots,a_n),\qquad a_i\in\{+1,-1\},\qquad \sum_{i=1}^n a_i=0.
\]
Then $\|\bA_n\|=1$, $\tr(\bA_n)=0$, $\tr(\bA_n^2)=n$,
$
\tr(\bA_n\bSig_n)=0,
$
and $Q_n(\bA_n)=r^\top \bA_n r=n\sum_{i=1}^n a_iY_i^2$.

Set
\[
\beta_4=\E(Y_1^4),\qquad \beta_{22}=\E(Y_1^2Y_2^2).
\]
Since $\sum_{i=1}^n Y_i^2=1$,
\[
1=\E\Big(\sum_{i=1}^n Y_i^2\Big)^2=n\beta_4+n(n-1)\beta_{22},\qquad \beta_{22}=\frac{1-n\beta_4}{n(n-1)}.
\]
Hence
\begin{align*}
\E\big(Q_n^2(\bA_n)\big)
&=n^2\Big(\sum_{i=1}^n a_i^2\E(Y_i^4)+\sum_{i\neq j}a_ia_j\E(Y_i^2Y_j^2)\Big).
\end{align*}
Using $\sum_{i\neq j}a_ia_j=(\sum_i a_i)^2-\sum_i a_i^2=-n$, we get
\[
\E\big(Q_n^2(\bA_n)\big)=n^2(n\beta_4-n\beta_{22})
=n^3\Big(\beta_4-\frac{1-n\beta_4}{n(n-1)}\Big)
=\frac{n^2(n^2\beta_4-1)}{n-1}.
\]
By definition of the fourth order correction kernel,
\[
\Var(Q_n(\bA_n))=2\tr(\bA_n^2)+\Gamma_n(\bA_n,\bA_n)=2n+\Gamma_n(\bA_n,\bA_n),
\]
hence
\[
\Gamma_n(\bA_n,\bA_n)=\frac{n^2(n^2\beta_4-1)}{n-1}-2n=n^3\beta_4+O(n).
\]
Now according to \cite{LiPanXieZhou2024},
\[
\beta_4\asymp n^{-\alpha/2}l(n^{1/2}),
\]
so
\[
\Gamma_n(\bA_n,\bA_n)\asymp n^{3-\alpha/2}l(n^{1/2}).
\]

We now test Condition~\ref{cond:g4} with $\bA_n=\bB_n=\diag(a_1,\dots,a_n)$ and $\widetilde{\bA}_n=0$. Since $\bA_n$ is deterministic and $\|\bA_n\|=1$, one has
\[
\|\bA_n-\widetilde{\bA}_n\|_{\mathrm{bil},L_1}=1.
\]
Thus the right-hand side in \eqref{eq:gamma_bilL1_cont1} is $o(n^{3/2})$, whereas
\[
\|\Gamma_n(\bA_n,\bA_n)\|_{L_1}=|\Gamma_n(\bA_n,\bA_n)|\asymp n^{3-\alpha/2}l(n^{1/2}).
\]
Therefore
\[
\frac{\|\Gamma_n(\bA_n,\bA_n)\|_{L_1}}{n^{3/2}}\asymp n^{3/2-\alpha/2}l(n^{1/2}),
\]
which implies
\[
\alpha<3\quad\Longrightarrow\quad \|\Gamma_n(\bA_n,\bA_n)\|_{L_1}\not=o(n^{3/2}).
\]
Hence Condition~\ref{cond:g4} fails for $\alpha<3$. In the critical case $\alpha=3$, the ratio is of order $l(n^{1/2})$, exactly matching the sharp necessary and sufficient condition
\[
\lim_{x\to\infty}x^3\Pr(|X|>x)=0
\]
from Theorem~1.6 of \cite{LiPanXieZhou2024}.

This shows that Condition~\ref{cond:g4} is close to optimal for a \emph{universal} CLT. Once Condition~\ref{cond:g4} fails, the fluctuation formula is no longer universally closed. Indeed, for the Schott statistic $f(x)=x^2$, \cite{LiPanXieZhou2024} proved
\[
\Var(\tr \widehat{\bR}^2)=2np^2\beta_4^2+4p^2n^{-2}+o(1).
\]
Thus the leading variance correction depends not only on the tail index $\alpha$, but also on the slowly varying function $l$. Without further assumptions on the tail behavior, one cannot in general expect a closed-form correction formula uniform across a whole model class. 

\subsubsection{Condition~\ref{cond:g5} and the exclusion of sparse spike mechanisms}
\label{subsubsec:g5_delta_necessary}

We next show that the vanishing factor $\Delta_n\downarrow0$ in Condition~\ref{cond:g5} is essential if one wants a universal Gaussian CLT at the original $O(1)$ scale. The point is that, after removing this factor, the remaining Rosenthal-type upper bound can still hold at the exact order of the right-hand side, while even the simplest linear spectral statistic already fails to be asymptotically normal.

Consider the isotropic sample covariance model
\[
\bS_n=\frac1n\sum_{j=1}^n \br_j\br_j^\top,\qquad \bSig_n=\bI_{p},\qquad \frac{p}{n}\to c\in(0,\infty),
\]
where the entries of $\br_j=(r_{1j},\dots,r_{pj})^\top$ are i.i.d.\ according to the following triangular-array distribution. Let
$
q_n=\frac{\lambda}{p n}, \lambda>0,
$
and define $r_{11}$ by
\[
\Pr(r_{11}=\pm \sqrt n)=\frac{q_n}{2},\qquad \Pr(r_{11}=\pm b_n)=\frac{1-q_n}{2},
\]
where $b_n>0$ is chosen so that $\E r_{11}^2=1$, equivalently,
$
(1-q_n)b_n^2+q_n n=1.
$
Since $q_n n=\lambda/p\to0$, one has
$
b_n^2=\frac{1-q_n n}{1-q_n}\to1.
$
Thus the coordinates are bounded by $\sqrt n$.

We test Condition~\ref{cond:g5} on the identity direction $\bA_n=\bI_{p}$. We have
\[
Q_n(\bI_{p})=\|\br_1\|^2-p=\sum_{i=1}^{p}(r_{i1}^2-1).
\]
Let
\[
K_n:=\sum_{i=1}^{p}\mathbf{1}_{\{|r_{i1}|=\sqrt n\}}\sim \mathrm{Bin}(p,q_n).
\]
One has $\Pr(K_n=1)\sim pq_n=\lambda/n$ and $\Pr(K_n\ge2)=O(n^{-2})$. Also,
\[
Q_n(\bI_{p})=K_n(n-1)+(p-K_n)(b_n^2-1).
\]
Since
$
b_n^2-1=-\frac{q_n(n-1)}{1-q_n},
$
we obtain
\[
p(b_n^2-1)=-\frac{pq_n(n-1)}{1-q_n}=-\frac{\lambda(n-1)}{n(1-q_n)}\to-\lambda.
\]
Hence
\[
Q_n(\bI_{p})=(n-b_n^2)K_n+p(b_n^2-1)=
\begin{cases}
-\lambda+o(1), & K_n=0,\\
n-\lambda+o(n), & K_n=1,\\
O(nK_n), & K_n\ge2.
\end{cases}
\]

Fix any integer $q\ge2$. The event $K_n=1$ dominates the $q$th moment, and therefore
\[
\E|Q_n(\bI_{p})|^q\asymp \frac{\lambda}{n}n^q=\lambda n^{q-1}.
\]
Because $\|\bI_{p}\|=1$, it follows that
$
n^{q-2}V_n^2(\bI_{p})\|\bI_{p}\|^{q-2}\asymp \lambda n^{q-1}.
$
Hence
\[
\E|Q_n(\bI_{p})|^q\asymp n^{q-2}V_n^2(\bI_{p})\|\bI_{p}\|^{q-2}.
\]
Thus, after removing the vanishing factor $\Delta_n$, the right-hand side in Condition~\ref{cond:g5} is attained at the exact order by the identity test matrix.

Now consider the linear spectral statistic corresponding to the test function $f(x)=x$,
\[
L_n:=\tr(\bS_n)-p=\frac1n\sum_{j=1}^n Q_n^{(j)}(\bI_{p}),
\]
where $Q_n^{(j)}(\bI_{p})$ are i.i.d.\ copies of $Q_n(\bI_{p})$. Let
\[
N_n:=\sum_{j=1}^n K_n^{(j)},
\]
the total number of spikes in the whole $p\times n$ data matrix. Since $pnq_n=\lambda$, we have
\[
N_n\sim \mathrm{Bin}(pn,q_n)\xrightarrow{d} \mathrm{Poisson}(\lambda).
\]
Using the representation above,
\[
L_n=\frac{n-b_n^2}{n}N_n+p(b_n^2-1).
\]
Since
\[
\frac{n-b_n^2}{n}\to1,\qquad p(b_n^2-1)\to-\lambda,
\]
it follows that
\[
L_n+\lambda\xrightarrow{d} \mathrm{Poisson}(\lambda).
\]

This gives a direct counterexample: if one removes the vanishing factor $\Delta_n$ from Condition~\ref{cond:g5}, then the resulting moment condition is too weak to guarantee the universal Gaussian CLT.

\subsubsection{The role of Condition~\ref{cond:g6}}

Condition~\ref{cond:g6} ensures that the LSS fluctuation remains a genuine bulk
phenomenon. Indeed, if escaping eigenvalues occur with nonnegligible
probability, then one can choose an analytic test function that vanishes on a
neighborhood of the deterministic bulk support but not at the escape location.
The resulting linear spectral statistic then receives an additional contribution
of order one from the outlier, which is not described by the bulk correction
mechanism represented by $\Gamma_n(\cdot,\cdot)$. Thus some form of spectral
confinement in probability is indispensable for a universal bulk fluctuation
theory.

The high probability form imposed in Condition~\ref{cond:g6} is stronger than
this minimal requirement. It is adopted because it matches the needs of the
contour, resolvent, martingale argument and guarantees that the contour may be
chosen uniformly with negligible error.

%\subsection{Scope of the framework}

%The purpose of the GHOST framework is to strike a workable balance between generality and structure. It is broad enough to capture the nonclassical fluctuation effects of interest, yet restrictive enough to yield explicit asymptotic formulas and usable statistical reductions. This is exactly the level of abstraction needed for the general theory, the blockwise realization, and the sphericity application developed below.

\section{Main results under the general framework}\label{sec:main_results}

\subsection{Limiting spectral distribution}
We first record the first order spectral limit. 

\begin{thm}\label{thm:lsd_general}
Assume Conditions~\ref{cond:g1}--\ref{cond:g4}. Then the empirical spectral distribution $F^{\bS_n}$ converges weakly, almost surely, to a nonrandom probability distribution $F^{c,H}$ whose Stieltjes transform $m(z)$ is the unique solution in $\C^+$ to
\begin{align}\label{eq:lsd_general}
m(z)=\int\frac{1}{t(1-c-czm(z))-z}\,dH(t).
\end{align}
Equivalently, the companion Stieltjes transform
\[
\underline m(z)=cm(z)-\frac{1-c}{z}
\]
satisfies
\begin{align}\label{eq:lsd_general_companion}
m(z)=-\frac1z\int\frac{1}{1+t\underline m(z)}\,dH(t).
\end{align}
\end{thm}

For later use, write
\[
m_n(z)=m_{F^{\bS_n}}(z).
\]
Let $m_n^0(z)$ denote the Stieltjes transform of the deterministic equivalent $F^{c_n,H_n}$ associated with $(c_n,H_n)$, and define
\[
\underline m_n^0(z)=c_n m_n^0(z)-\frac{1-c_n}{z}.
\]

\subsection{LSS fluctuation theorem}
We now state the main fluctuation theorem. It shows that, within the GHOST regime, the fluctuation of linear spectral statistics remains asymptotically Gaussian, and that all model dependence enters through the effective fourth order correction kernel $\Gamma_n(\cdot,\cdot)$.

Let $f$ be analytic on an open neighborhood containing the confinement interval $\mathcal I$ in Condition~\ref{cond:g6}. Define the linear spectral statistic
\[
L_n(f)=p\int f(x)\,d\big(F^{\bS_n}(x)-F^{c_n,H_n}(x)\big).
\]
%Its fluctuation is encoded by the resolvent process
%%\[
%M_n(z)=p\big(m_n(z)-m_n^0(z)\big).
%\]

For $z\in\C^+$, define
\[
\bH_n(z)=\underline m_n^0(z)\bSig_n+\bI_p.
\]

For an analytic test function $f$, define the mean approximants
\begin{align}
\mathcal M_{0,n}(f)
&=
-\frac1{2\pi i}\oint_{\mathcal C}
f(z)\,
\frac{
c_n\big(\underline m_n^0(z)\big)^3
\displaystyle\int\frac{t^2}{\big(1+\underline m_n^0(z)t\big)^3}\,dH_n(t)
}{
\Big(
1-c_n\displaystyle\int
\frac{\big(\underline m_n^0(z)\big)^2t^2}
{\big(1+\underline m_n^0(z)t\big)^2}\,dH_n(t)
\Big)^2
}
\,dz, \label{eq:M0n_main}\\
\mathcal M_{1,n}(f)
&=
-\frac1{2\pi i}\oint_{\mathcal C}
f(z)\,
\frac{
\big(\underline m_n^0(z)\big)^3
}{
1-c_n\displaystyle\int
\frac{\big(\underline m_n^0(z)\big)^2t^2}
{\big(1+\underline m_n^0(z)t\big)^2}\,dH_n(t)
}
\cdot
\frac1n
\Gamma_n\Big(
\bH_n^{-1}(z),\,
\bH_n^{-2}(z)
\Big)
\,dz. \label{eq:M1n_main}
\end{align}
Set
\begin{align}
\mathcal M_n(f)=\mathcal M_{0,n}(f)+\mathcal M_{1,n}(f). \label{eq:Mn_main}
\end{align}

For analytic test functions $f$ and $g$, define the covariance approximants
\begin{align}
\mathcal V_{0,n}(f,g)
&=
-\frac1{4\pi^2}
\oint_{\mathcal C_1}\oint_{\mathcal C_2}
f(z_1)g(z_2)\,
2\Bigg(
\frac{(\underline m_n^0)'(z_1)(\underline m_n^0)'(z_2)}
{\big(\underline m_n^0(z_1)-\underline m_n^0(z_2)\big)^2}
-\frac1{(z_1-z_2)^2}
\Bigg)\,dz_1dz_2, \label{eq:V0n_main}\\
\mathcal V_{1,n}(f,g)
&=
-\frac1{4\pi^2}
\oint_{\mathcal C_1}\oint_{\mathcal C_2}
f(z_1)g(z_2)\,
(\underline m_n^0)'(z_1)(\underline m_n^0)'(z_2)\,
\frac1n\Gamma_n\big(\bH_n^{-1}(z_1),\bH_n^{-1}(z_2)\big)\,
dz_1dz_2. \label{eq:V1n_main}
\end{align}
Set
\begin{align}
\mathcal V_n(f,g)=\mathcal V_{0,n}(f,g)+\mathcal V_{1,n}(f,g). \label{eq:Vn_main}
\end{align}

\begin{thm}\label{thm:lss_general_gamma}
Assume that the sample vectors $\br_1,\dots,\br_n$ belong to the GHOST class. Then for any analytic test functions $f_1,\dots,f_s$,
\[
\big(L_n(f_1),\dots,L_n(f_s)\big)
\]
is asymptotically Gaussian with finite sample mean approximant
\[
\bm\mu_n=\big(\mathcal M_n(f_1),\dots,\mathcal M_n(f_s)\big)^\top
\]
and covariance approximant
\[
\bm\Omega_n=\big(\mathcal V_n(f_j,f_k)\big)_{1\le j,k\le s}.
\]
In particular, along any subsequence on which $\bm\mu_n$ and $\bm\Omega_n$ converge, the corresponding statistics converge in distribution to a Gaussian vector with that limiting mean and covariance.
\end{thm}

\begin{remark}[On the fluctuation scale]\label{rem:scale_phase_transition}
The theorem is formulated under the classical normalization
\[
M_n(z)=p\big(m_n(z)-m_n^0(z)\big),
\]
which is the natural scale in the universal Gaussian regime isolated by the GHOST conditions. If the contribution of $\Gamma_n$ is larger than the classical covariance scale, then the samecontour integration, resolvent expansion, and martingale decomposition scheme can still be carried out after rescaling by the square root of the dominant excess order. 

We do not formulate this scale-adaptive version in full generality here, since the required normalization depends on the model specific order of $\Gamma_n$. This mechanism is seen most clearly in the blockwise mixed radial model of Section~\ref{sec:block}, where $\Gamma_n$ is explicit and the fluctuation-scale transition can be read off directly.
\end{remark}

The proof follows the classical contour integration, resolvent expansion, and martingale decomposition scheme, with the new input that the fourth order quadratic form correction is tracked through the covariance and centering calculations. We defer the technical proof details and several auxiliary calculations to the Supplementary Material, and verify the full program below in a concrete model.

\section{A blockwise mixed radial realization of the GHOST framework}\label{sec:block}

We now present a concrete heterogeneous model that realizes the GHOST mechanism in explicit form. The model is important for two reasons. First, it verifies that the abstract conditions are nonempty and lead to tractable fluctuation formulas beyond classical independent component settings. Second, it exhibits the fluctuation scale transition generated by the effective fourth order correction.

\subsection{Model specification}

Let
\[
\br=\bga_n\bw
=
\begin{pmatrix}
\bga_x & \mb 0\\
\mb 0 & \bga_y
\end{pmatrix}
\bw,
\qquad
\bw=
\begin{pmatrix}
\rho_x\bx\\
\rho_y\by
\end{pmatrix},
\]
where $\bx\in\R^{p_1}$ and $\by\in\R^{p_2}$ are independent random vectors, $\rho_x$ and $\rho_y$ are nonnegative scalar random variables, and
\[
p_1+p_2=p.
\]

We impose the following assumptions.

\begin{itemize}
\item[(i)] $\bx$, $\by$, $\rho_x$, and $\rho_y$ are  independent;

\item[(ii)] the coordinates of $\bx$ are i.i.d.\ with
\[
\E x_1=0,\qquad \E x_1^2=1,\qquad \E x_1^4=\nu_{x4}+3,\qquad \E x_1^6<\infty;
\]

\item[(iii)] the coordinates of $\by$ are i.i.d.\ with
\[
\E y_1=0,\qquad \E y_1^2=1,\qquad \E y_1^4=\nu_{y4}+3,\qquad \E y_1^6<\infty;
\]

\item[(iv)] the radial variables satisfy
\[
\E \rho_x^2=\E \rho_y^2=1,\qquad \Var(\rho_x^2)=\tau_x p^{-\delta_1}(1+o(1)),
\qquad
\Var(\rho_y^2)=\tau_y p^{-\delta_2}(1+o(1)),
\]
for constants $\tau_x,\tau_y\ge 0$ and $\delta_1,\delta_2>0$;

\item[(v)] $\bga_n$ is a deterministic $p\times p$ matrix, and the empirical spectral distribution of
\[
\bSig_n=\bga_n\bga_n^\top
\]
converges weakly to a nondegenerate probability distribution $H$. Moreover, $\|\bSig_n\|$ is uniformly bounded in $n$;

\item[(vi)] the block dimensions satisfy
\[
c_{nx}=p_1/p^{\alpha_1}\to c_x\in(0,\infty),
\qquad
c_{ny}=p_2/p^{\alpha_2}\to c_y\in(0,\infty),
\]
for $\alpha_1,\alpha_2\in[0,1]$, and
\[
c_n=\frac{p}{n}\to c\in(0,\infty).
\]
\end{itemize}

We also introduce the block-selection matrices
\[
\bF_{p_1}^1=
\begin{pmatrix}
\bI_{p_1} & \mb 0\\
\mb 0 & \mb 0
\end{pmatrix},
\qquad
\bF_{p_2}^2=
\begin{pmatrix}
\mb 0 & \mb 0\\
\mb 0 & \bI_{p_2}
\end{pmatrix}.
\]

This model contains the classical independent component regime as the special
case in which the radial variables are degenerate. More generally, it allows
local fourth moment heterogeneity and blockwise energy fluctuation to coexist
within one framework. It therefore provides a concrete realization of the GHOST
mechanism in which the correction kernel becomes explicit.

The two block formulation is already minimal for the purposes of the present theory. It is the simplest setting in which the effective fourth order correction contains both a coordinatewise Hadamard sector and a block energy trace product sector. These are precisely the two structural components that drive the nonclassical behavior in the GHOST framework.

\subsection{Explicit form of the quadratic form correction}
In this model, the abstract fourth order kernel admits a closed form, so the general fluctuation theorem becomes fully explicit. The resulting formula reveals directly how different structural sources of non-Gaussianity enter the mean, covariance, and fluctuation scale.
Let
\[
\kappa=\frac12\max\{2\alpha_1-\delta_1,\ 2\alpha_2-\delta_2\}<1,
\qquad
r_p(\kappa)=\min\{1,p^{1/2-\kappa}\}.
\]
Define
\[
G_n(x)=pr_p(\kappa)\big(F^{\bS_n}(x)-F^{c_n,H_n}(x)\big).
\]

\begin{thm}[Blockwise specialization of the GHOST fluctuation theorem]\label{thm2}
Let $f_1,\dots,f_s$ be analytic on an open region containing the interval
\begin{align}\label{int}
\left[
\liminf_n \lambda_{\min}^{\bSig_n} I_{(0,1)}(c)(1-\sqrt c)^2,\,
\limsup_n \lambda_{\max}^{\bSig_n}(1+\sqrt c)^2
\right].
\end{align}
Assume the blockwise mixed radial model, and suppose
\[
\E\big|p^{\delta_1/2}(\rho_{x1}^2-1)\big|^{\frac{2}{\min\{\delta_1,1\}}}<\infty,
\qquad
\E\big|p^{\delta_2/2}(\rho_{y1}^2-1)\big|^{\frac{2}{\min\{\delta_2,1\}}}<\infty.
\]
Then
\[
\left(\int f_1(x)\,dG_n(x),\dots,\int f_s(x)\,dG_n(x)\right)
\xrightarrow{d}
(X_{f_1},\dots,X_{f_s}),
\]
where $(X_{f_1},\dots,X_{f_s})$ is a Gaussian vector.

Moreover, its mean and covariance are given by the general formulas in Theorem~\ref{thm:lss_general_gamma}, with the abstract correction functional $\Gamma_n$ replaced by
\begin{align}
\Gamma_n^{\mathrm{blk}}(\bA,\bB)
=
&\ \nu_{x4}\tr\big[(\bA\circ\bB)\bF_{p_1}^1\big]
+\nu_{y4}\tr\big[(\bA\circ\bB)\bF_{p_2}^2\big] \notag\\
&+\tau_x p^{-\delta_1}\tr(\bA\bF_{p_1}^1)\tr(\bB\bF_{p_1}^1)
+\tau_y p^{-\delta_2}\tr(\bA\bF_{p_2}^2)\tr(\bB\bF_{p_2}^2). \label{eq:block_gamma_thm}
\end{align}
\end{thm}

\begin{remark}
The assumptions
\[
\E\big|p^{\delta_1/2}(\rho_{x1}^2-1)\big|^{\frac{2}{\min\{\delta_1,1\}}}<\infty,
\qquad
\E\big|p^{\delta_2/2}(\rho_{y1}^2-1)\big|^{\frac{2}{\min\{\delta_2,1\}}}<\infty
\]
correspond to the quadratic form high moment condition in the abstract theory. In the present model, centered quadratic form fluctuations contain both coordinatewise and block energy components, and the above conditions control the latter. The distinction between the cases $\delta_j\ge 1$ and $\delta_j<1$ reflects the
higher moment requirement created by slower decay of block energy fluctuation.
\end{remark}

Thus the blockwise mixed radial model gives an explicit realization of the GHOST decomposition: the effective fourth order correction splits into a Hadamard type sector driven by coordinatewise fourth moments and a trace product sector driven by blockwise energy fluctuations.

The same decomposition persists for finite multi block extensions, where the
non-Gaussian correction remains a superposition of within block fourth moment
terms and block level energy fluctuation terms. 

\subsection{Phase transition in the fluctuation scale}

The entrywise fourth moment terms in \eqref{eq:block_gamma_thm} remain on the classical order $p$. By contrast, the radial trace product terms can be weaker than, comparable to, or stronger than the classical scale, depending on the relative size of the block dimensions and the decay rate of the radial variances. This variation in order is exactly what drives the phase transition in the fluctuation scale.

Define
\[
\phi_x=2\alpha_1-\delta_1,
\qquad
\phi_y=2\alpha_2-\delta_2.
\]
These indices quantify the effective strength of the two blockwise radial corrections. If $\phi_x>\phi_y$, then the first block dominates; if $\phi_y>\phi_x$, then the second block dominates; and if $\phi_x=\phi_y$, then both contribute on the same asymptotic scale.

If
\[
\max\{\phi_x,\phi_y\}<1,
\]
then the radial correction is asymptotically no stronger than the classical scale and
\[
r_p(\kappa)=1,
\]
so the fluctuation remains on the classical order. If instead
\[
\max\{\phi_x,\phi_y\}>1,
\]
then the radial correction dominates the classical fluctuation scale and
\[
r_p(\kappa)=p^{1/2-\kappa},
\]
so a slower, nonclassical normalization is required.

Thus the blockwise model separates two distinct thresholds: a fluctuation threshold, at which the normalization changes while the limiting spectral law remains classical, and a later spectral threshold, at which the bulk law itself ceases to be of classical sample covariance type.

\subsection{Why standard calibrations can miss the relevant correction}

The explicit kernel \eqref{eq:block_gamma_thm} shows that existing calibrations fail for a structural reason: they typically capture only one low-complexity sector of the effective fourth order correction.

In independent component models, the correction is of Hadamard type and is driven by coordinatewise fourth moments. In elliptical models, the correction is of trace-product type and is driven by radial energy fluctuation. By contrast, the blockwise mixed radial model contains both sectors simultaneously:
\[
\nu_{x4}\tr\big[(\bA\circ\bB)\bF_{p_1}^1\big]
+\nu_{y4}\tr\big[(\bA\circ\bB)\bF_{p_2}^2\big]
\]
and
\[
\tau_x p^{-\delta_1}\tr(\bA\bF_{p_1}^1)\tr(\bB\bF_{p_1}^1)
+\tau_y p^{-\delta_2}\tr(\bA\bF_{p_2}^2)\tr(\bB\bF_{p_2}^2).
\]
Therefore, a calibration based only on entrywise fourth moments misses the block-energy sector, while a calibration derived from elliptical models need not recover the Hadamard sector and may also miss stronger energy-driven effects once the fluctuation scale leaves the classical regime.

This is precisely why the relevant object is not a model specific correction formula, but the effective bilinear kernel $\Gamma_n(\cdot,\cdot)$ itself.

\subsection{Numerical illustration of the phase transition}

We now provide a numerical illustration of the fluctuation scale transition predicted by the blockwise mixed radial model. .

For simplicity, we consider a one block dominant design with
\[
\alpha=1,\qquad \tau=1,\qquad c=\frac{p}{n}=0.5,
\]
and Gaussian directional vectors. In this case,
\[
\phi=2\alpha-\delta=2-\delta,
\qquad
\kappa=\frac{\phi}{2},
\qquad
r_p(\kappa)=
\begin{cases}
1, & \phi\le 1,\\[0.4em]
p^{(1-\phi)/2}, & \phi>1.
\end{cases}
\]
We vary $\delta$ so that $\phi$ ranges from $0.6$ to $1.5$, thereby covering both the classical regime $\phi\le 1$ and the supercritical regime $\phi>1$.

For each configuration, we compute the empirical variance of the raw statistic and of its rescaled version over the grid
\[
p\in\{100,150,200,300,400,600,800\},\qquad n=2p.
\]

Figure~\ref{fig:phase_transition_fixedp} displays the resulting phase transition based on 10,000 Monte Carlo repetitions. The left panel plots the raw variance against $\phi$ for several fixed dimensions $p$, while the right panel plots the variance after scaling by $r_p(\kappa)$.

\begin{figure}[!htbp]
\centering
\includegraphics[width=0.95\textwidth]{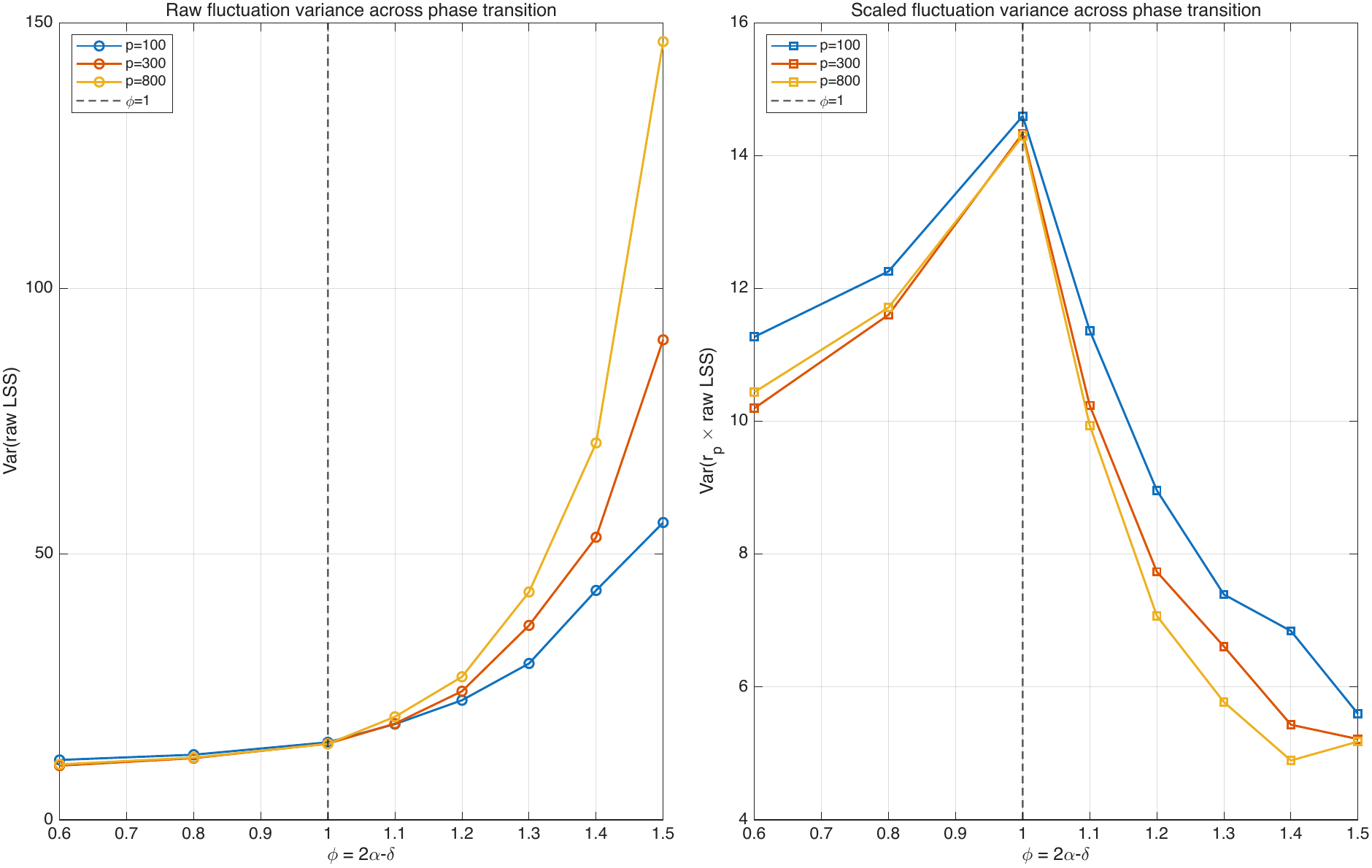}
\caption{Empirical variance across the phase transition parameter $\phi=2\alpha-\delta$ for several fixed dimensions $p$. Left: raw variance. Right: variance after phase-dependent scaling by $r_p(\kappa)$. The dashed vertical line marks the critical boundary $\phi=1$.}
\label{fig:phase_transition_fixedp}
\end{figure}

The raw variance remains relatively stable when $\phi\le 1$, but increases rapidly once $\phi>1$. After rescaling, the growth is largely removed. This agrees with the theoretical prediction that $\phi=1$ is the critical boundary and that $r_p(\kappa)$ captures the correct leading fluctuation order beyond that boundary.

\section{Sphericity testing under the GHOST framework}\label{sec:sphericity}

Testing the sphericity of a population covariance matrix is a classical problem in multivariate analysis. The null hypothesis is
\[
H_0:\quad \bSig_n=\sigma_n^2\bI_p,
\]
and the main issue in high dimensional settings is correct null calibration of eigenvalue based statistics; see, for example, \citet{LiY16Ta,OnatskiM13A,WangY13S,ZouP14M,Hu2019}.

Under the GHOST framework, the spherical null leads to a substantial simplification: although the general fluctuation theory depends on the bilinear functional $\Gamma_n(\cdot,\cdot)$, under sphericity the relevant correction is summarized by a single scalar quantity. This scalar reduction is what makes a feasible correction possible.

\subsection{Scalar reduction of the correction functional}

Recall that
\[
\bH_n(z)=\underline m_n^0(z)\bSig_n+\bI_p.
\]
Under
$
H_0,
$
this matrix reduces to
\[
\bH_n(z)=\bigl(1+\sigma_n^2\underline m_n^0(z)\bigr)\bI_p.
\]
Hence
\[
\bH_n^{-1}(z)=a_n(z)\bI_p,
\qquad
\bH_n^{-2}(z)=a_n^2(z)\bI_p,
\qquad
a_n(z)=\frac{1}{1+\sigma_n^2\underline m_n^0(z)}.
\]

By bilinearity,
\[
\Gamma_n\big(\bH_n^{-1}(z_1),\bH_n^{-1}(z_2)\big)
=
a_n(z_1)a_n(z_2)\Gamma_n(\bI_p,\bI_p),
\]
and similarly
\[
\Gamma_n\big(\bH_n^{-1}(z),\bH_n^{-2}(z)\big)
=
a_n^3(z)\Gamma_n(\bI_p,\bI_p).
\]
Therefore the entire model dependence enters through the single scalar
\[
\gamma_n:=\Gamma_n(\bI_p,\bI_p).
\]
This scalar reduction is specific to the spherical null. Outside this setting, the full bilinear structure of $\Gamma_n$ generally remains relevant.

By Lemma~\ref{lem:gamma_cumulant},
\[
\gamma_n
=
\E\Big[\big(\|\br_1\|^2-p\sigma_n^2\big)^2\Big]-2p\sigma_n^4.
\]
Since John's statistic is scale invariant, we may take $\sigma_n^2=1$ in what follows. Then
\[
\gamma_n
=
\E\Big[\big(\|\br_1\|^2-p\big)^2\Big]-2p.
\]
Thus $\gamma_n$ measures the excess fluctuation of the sample energy beyond the Gaussian benchmark.

\begin{remark}
The scalar reduction explains why corrected procedures derived under rather different structural assumptions may end up having the same feasible form under the spherical null: once the relevant fluctuation correction depends only on $\Gamma_n(\bI_p,\bI_p)$, any valid correction must target the same excess energy fluctuation.
\end{remark}

\subsection{Specialization to John's statistic}

Let
\[
L_{1n}=\tr(\bS_n),
\qquad
L_{2n}=\tr(\bS_n^2),
\]
which correspond to
\[
f_1(x)=x,
\qquad
f_2(x)=x^2.
\]
John's statistic is
\[
U_{Jn}=\frac{pL_{2n}}{L_{1n}^2}-1,
\qquad
Q_{Jn}=\frac{np}{2}U_{Jn}.
\]

Under the spherical null,
\[
L_{1n}=p+L_n(f_1),
\qquad
L_{2n}=p(1+c_n)+L_n(f_2),
\]
where
\[
L_n(f)=p\int f(x)\,d\bigl(F^{\bS_n}(x)-F^{c_n,H_n}(x)\bigr).
\]

The next proposition extracts the only model dependent first order correction needed for John's statistic. Its proof is a direct residue calculation based on \eqref{eq:M1n_main} under the spherical null and is included in the Supplementary Material together with other technical derivations.

\begin{proposition}\label{prop:john_bias_terms}
Under
$
H_0,
$
the model dependent mean correction in Theorem~\ref{thm:lss_general_gamma} satisfies
\[
\mathcal M_{1,n}(f_1)=0,
\qquad
\mathcal M_{1,n}(f_2)=-\frac{\gamma_n}{n}.
\]
\end{proposition}

Thus the correction vanishes for $L_{1n}$ and enters only through $L_{2n}$.

\subsection{Asymptotic centering and variance of John's statistic}

Let
\[
g(x,y)=\frac{py}{x^2}-1,
\qquad
U_{Jn}=g(L_{1n},L_{2n}).
\]
Applying the delta method to the joint fluctuation of $(L_{1n},L_{2n})$, together with Proposition~\ref{prop:john_bias_terms}, yields
\begin{align}
\E(nU_{Jn})
&=
p+1+\frac{\gamma_n}{p}+o(1),
\label{eq:john_mean_general_new}\\
\E(Q_{Jn})
&=
\frac12\bigl(p^2+p+\gamma_n\bigr)+o(1).
\label{eq:john_mean_Q_general_new}
\end{align}

For the variance, the model dependent correction in the first order term cancels at the delta method level for John's statistic, so that
\begin{align}
\Var(nU_{Jn})
&=
4+o(1),
\label{eq:john_var_general_new}\\
\Var(Q_{Jn})
&=
p^2+o(1).
\label{eq:john_var_Q_general_new}
\end{align}

Let
\[
T_i=\|\br_i\|^2.
\]
A natural estimator of the scalar correction $\gamma_n$ is 
\[
\widehat\gamma_n
=
\frac1n\sum_{i=1}^n (T_i-p)^2-2p.
\]

This estimator directly matches the probabilistic interpretation of $\gamma_n$ and avoids estimation of any high dimensional fourth order tensor or large collection of coordinatewise cumulants.

It is therefore natural to define
\begin{align}\label{eq:john_corr_general_new}
T_{Jn}^{\mathrm{corr}}
=
\frac{
nU_{Jn}-p-1-\widehat\gamma_n/p
}{2}.
\end{align}
Equivalently,
\begin{align}\label{eq:john_corr_Q_general_new}
\widetilde T_{Jn}^{\mathrm{corr}}
=
\frac{
Q_{Jn}-\dfrac12\bigl(p^2+p+\widehat\gamma_n\bigr)
}{p}.
\end{align}

The consistency of $\widehat\gamma_n/p$ follows directly from the quadratic form
moment regularity condition. Indeed, for $T_i=\|\br_i\|^2$, we have
\[
T_i-p=\br_i^\top \bI_p\br_i-\tr(\bSig_n),
\]
which is precisely the centered quadratic form in Condition~\ref{cond:g5} with $\bA_n=\bI_p$. The required law of large numbers for the sample second central moment then follows from the resulting fourth moment bound.

\begin{thm}\label{thm:john_general_new}
Assume the conditions of Theorem~\ref{thm:lss_general_gamma} hold and suppose $H_0$ is true. Then
\[
T_{Jn}^{\mathrm{corr}}\xrightarrow{d}N(0,1).
\]
Equivalently,
\[
\widetilde T_{Jn}^{\mathrm{corr}}\xrightarrow{d}N(0,1).
\]
\end{thm}

\section{Numerical studies}\label{sec:numerics}

In this section, we examine the finite sample consequences of the GHOST correction principle developed in Section~\ref{sec:sphericity}. 
 We compare three procedures based on 10,000 Monte Carlo repetitions:

\begin{itemize}
\item[(i)] the \emph{Gaussian calibration}, which uses the classical Gaussian centering and variance;

\item[(ii)] the \emph{Wang--Yao-type fourth moment correction}, following \citet{WangY13S}, which adjusts the centering through an entrywise fourth moment estimator while retaining the Gaussian variance normalization;

\item[(iii)] the \emph{proposed GHOST-based correction}, which estimates the scalar excess energy parameter $\gamma_n$ through $\widehat\gamma_n$.
\end{itemize}

\subsection{Simulation design}

We test the spherical null hypothesis
\[
H_0:\quad \bSig_n=\sigma^2\bI_p,
\]
and generate independent observations
\[
\br_i=\sigma \bw_i,\qquad i=1,\dots,n,
\]
where
\[
\E\bw_i=\mathbf 0,
\qquad
\E(\bw_i\bw_i^\top)=\bI_p.
\]

We consider six representative models. Model~1 is the Gaussian benchmark. Model~2 is a non-Gaussian independent component model. Models~3--6 introduce progressively stronger forms of heterogeneous radial or blockwise energy fluctuation. The six models are summarized in Table~\ref{tab:model_definitions}.

\begin{table}[H]
\centering
\caption{Representative models used in the numerical study. Models 1--2 serve as benchmark cases, while Models 3--6 introduce progressively stronger heterogeneous energy fluctuation beyond the independent-component setting.}
\label{tab:model_definitions}
\begin{tabular}{llllll}
\hline
Model & Structure & Direction & Ratios & $\tau$ & $\delta$ \\
\hline
M1 & 1-block & gaussian & -- & -- & -- \\
M2 & 1-block & rademacher & -- & -- & -- \\
M3 & 2-block & gaussian,gaussian & [0.4,0.6] & [0.8,1.2] & [1,0.8] \\
M4 & 2-block & gaussian,gaussian & [0.2,0.8] & [1,2] & [1,0.6] \\
M5 & 2-block & t8,t8 & [0.2,0.8] & [1,2] & [1,0.6] \\
M6 & 3-block & rademacher,rademacher,rademacher & [0.1,0.2,0.7] & [0.8,1.5,2.2] & [1.2,0.9,0.6] \\
\hline
\end{tabular}
\end{table}

For the null size experiment, we consider
\[
p\in\{50,100,150,200,250,300\},
\qquad
n=2p,
\]
so that the aspect ratio remains fixed at $p/n=0.5$. The proposed corrected statistic is
\[
T_{Jn}^{\mathrm{corr}}
=
\frac{
nU_{Jn}-p-1-\widehat\gamma_n/p
}{2}.
\]

For the power experiment, we consider a block variance alternative
\[
\bSig_n
=
\sigma^2\operatorname{diag}\bigl(
\underbrace{a,\dots,a}_{\lfloor 0.2p\rfloor},
1,\dots,1
\bigr),
\]
where $a>1$ controls the signal strength. Since the Gaussian and fourth moment corrected procedures can be substantially mis-calibrated under the null in heterogeneous models, we report \emph{size-adjusted power}, using empirical null critical values computed separately for each method.

\subsection{Empirical size under the spherical null}

We first study empirical size under the spherical null. Table~\ref{tab:size_results_p200} reports rejection rates at nominal level $5\%$ for the representative setting $(p,n)=(200,400)$.

\begin{table}[H]
\centering
\caption{Empirical rejection rates at the nominal $5\%$ level for the Gaussian calibration, the fourth moment correction, and the proposed corrected procedure, with $(p,n)=(200,400)$.}
\label{tab:size_results_p200}
\begin{tabular}{lccc}
\hline
Model & Gaussian size & WY size & Corrected size \\
\hline
M1 & 0.0485 & 0.0485 & 0.0455 \\
M2 & 0.1735 & 0.0535 & 0.0550 \\
M3 & 0.1045 & 0.1015 & 0.0470 \\
M4 & 1.0000 & 1.0000 & 0.0635 \\
M5 & 1.0000 & 0.9995 & 0.0585 \\
M6 & 0.9325 & 0.9925 & 0.0575 \\
\hline
\end{tabular}
\end{table}

Under the Gaussian benchmark (M1), all three methods are well calibrated. Under the independent component non-Gaussian model (M2), the Gaussian calibration shows noticeable distortion, whereas both the fourth moment correction and the proposed procedure remain close to the nominal level. Once heterogeneous energy fluctuation is introduced, the pattern changes. In Models~3--6, the Gaussian and fourth moment procedures become increasingly distorted, while the proposed correction remains much more stable.

This agrees with the theory. In the heterogeneous models, the dominant null distortion is no longer described by a single homogeneous fourth moment effect. Instead, it is driven by excess fluctuation of the sample energy, which is precisely what $\widehat\gamma_n$ is designed to capture.

To display the full pattern across dimensions, Figure~\ref{fig:size_heatmap} plots the size distortion
\[
\bigl|\widehat{\mathrm{size}}-0.05\bigr|
\]
for all six models and all six dimension settings.

\begin{figure}[!htbp]
\centering
\includegraphics[width=0.95\textwidth]{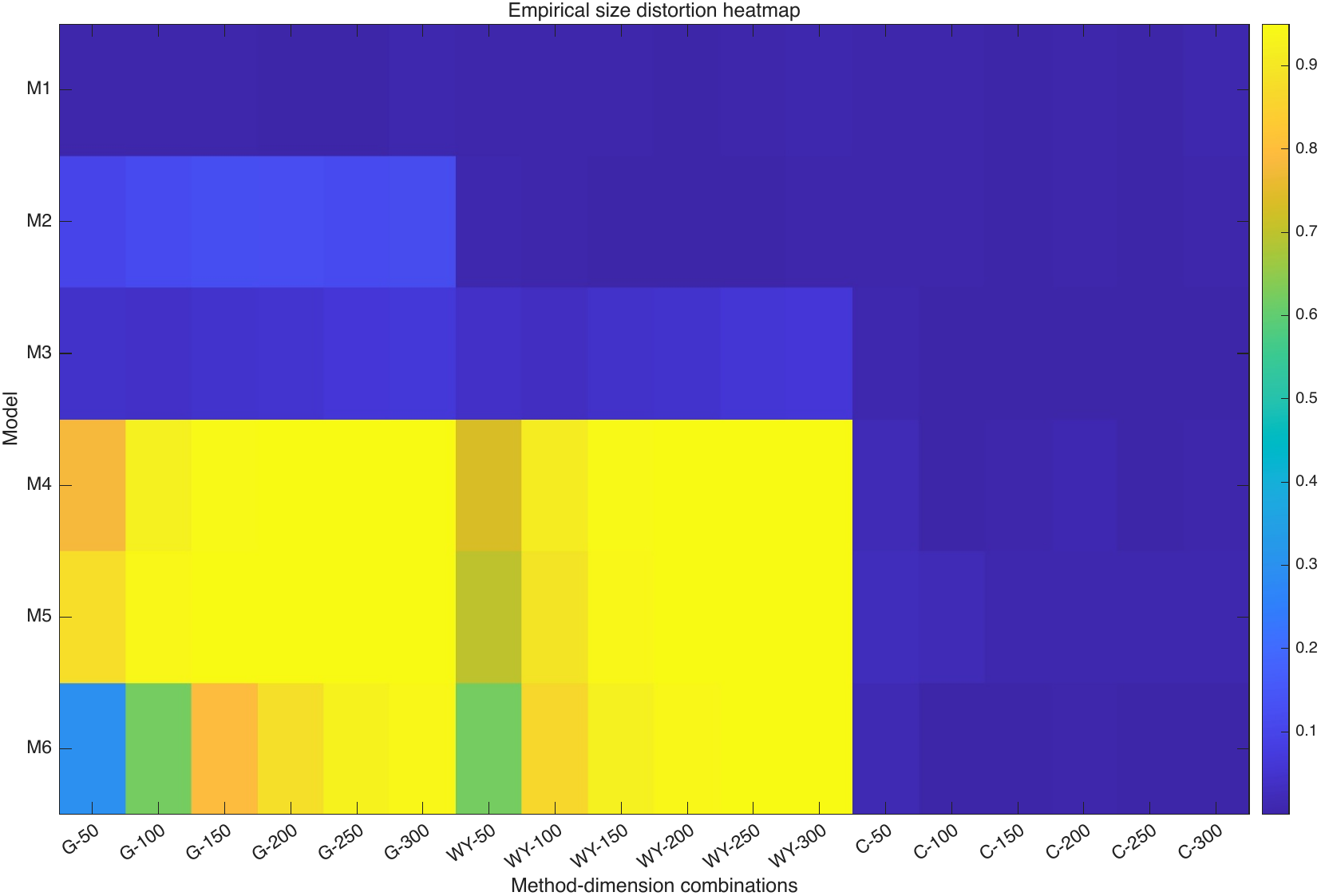}
\caption{Heatmap of empirical size distortion $|\widehat{\mathrm{size}}-0.05|$
across the six representative models and the dimension settings
$p=50,100,150,200,250,300$ with $n=2p$. Lighter colors indicate better null calibration.}
\label{fig:size_heatmap}
\end{figure}

For visual comparison at $(p,n)=(200,400)$, Figure~\ref{fig:size_models_bar} displays the rejection rates of the three procedures across the six models.

\begin{figure}[!htbp]
\centering
\includegraphics[width=0.95\textwidth]{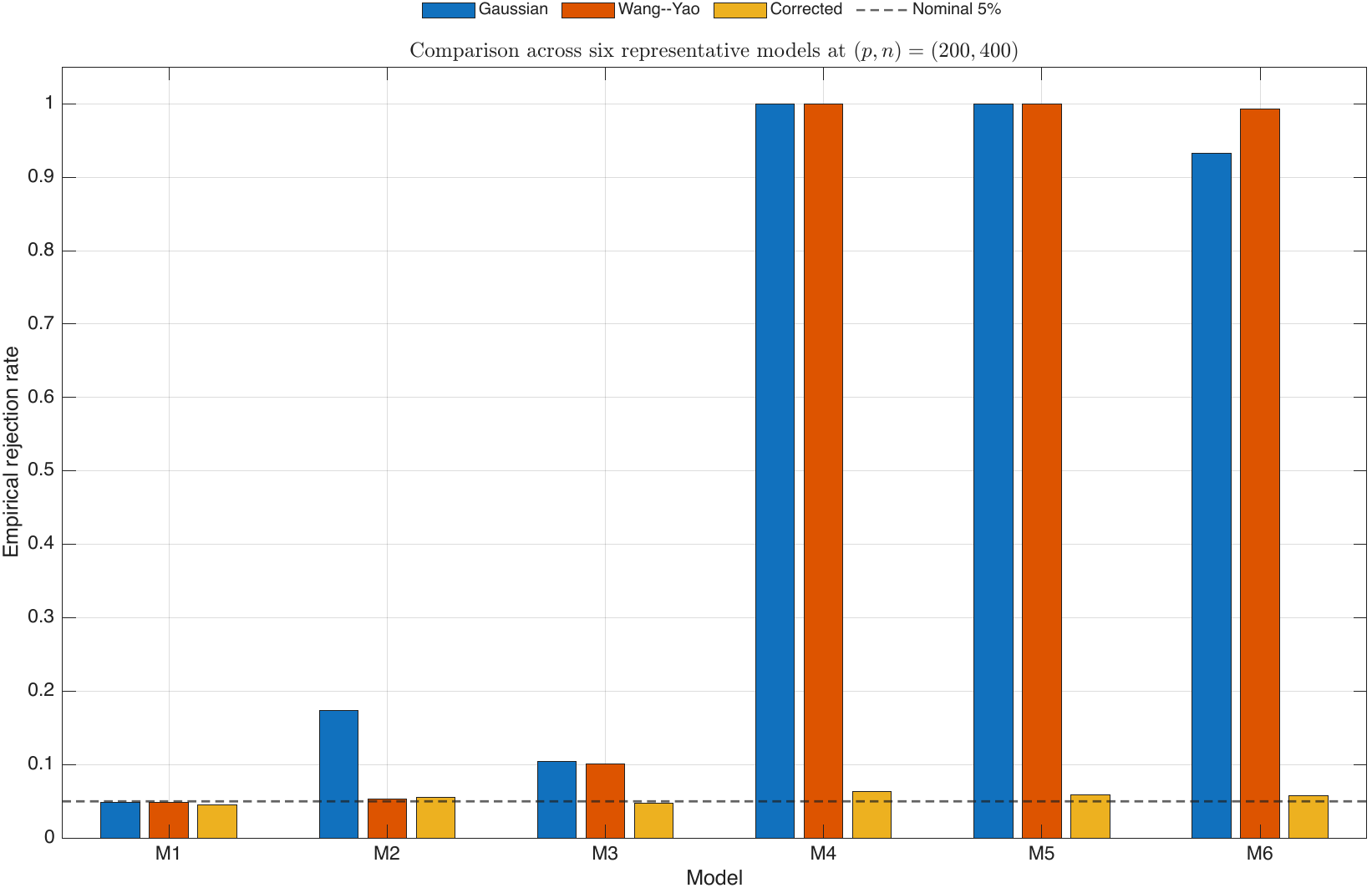}
\caption{Empirical rejection rates of the Gaussian calibration, the Wang--Yao-type fourth moment correction, and the proposed corrected procedure across the six representative models at $(p,n)=(200,400)$. The dashed horizontal line marks the nominal $5\%$ level.}
\label{fig:size_models_bar}
\end{figure}

These results show that under classical homogeneous settings the existing calibrations remain adequate, but once heterogeneous energy fluctuation becomes non-negligible, accurate null calibration requires a correction that targets that mechanism directly.

\subsection{Size-adjusted power under block variance alternatives}

We next study size-adjusted power. Figure~\ref{fig:power_curves} reports power curves at $(p,n)=(200,400)$ for three representative models: the independent component benchmark M2, the strong two-block heterogeneous model M4, and the three-block heterogeneous model M6.

\begin{figure}[H]
\centering
\includegraphics[width=0.95\textwidth]{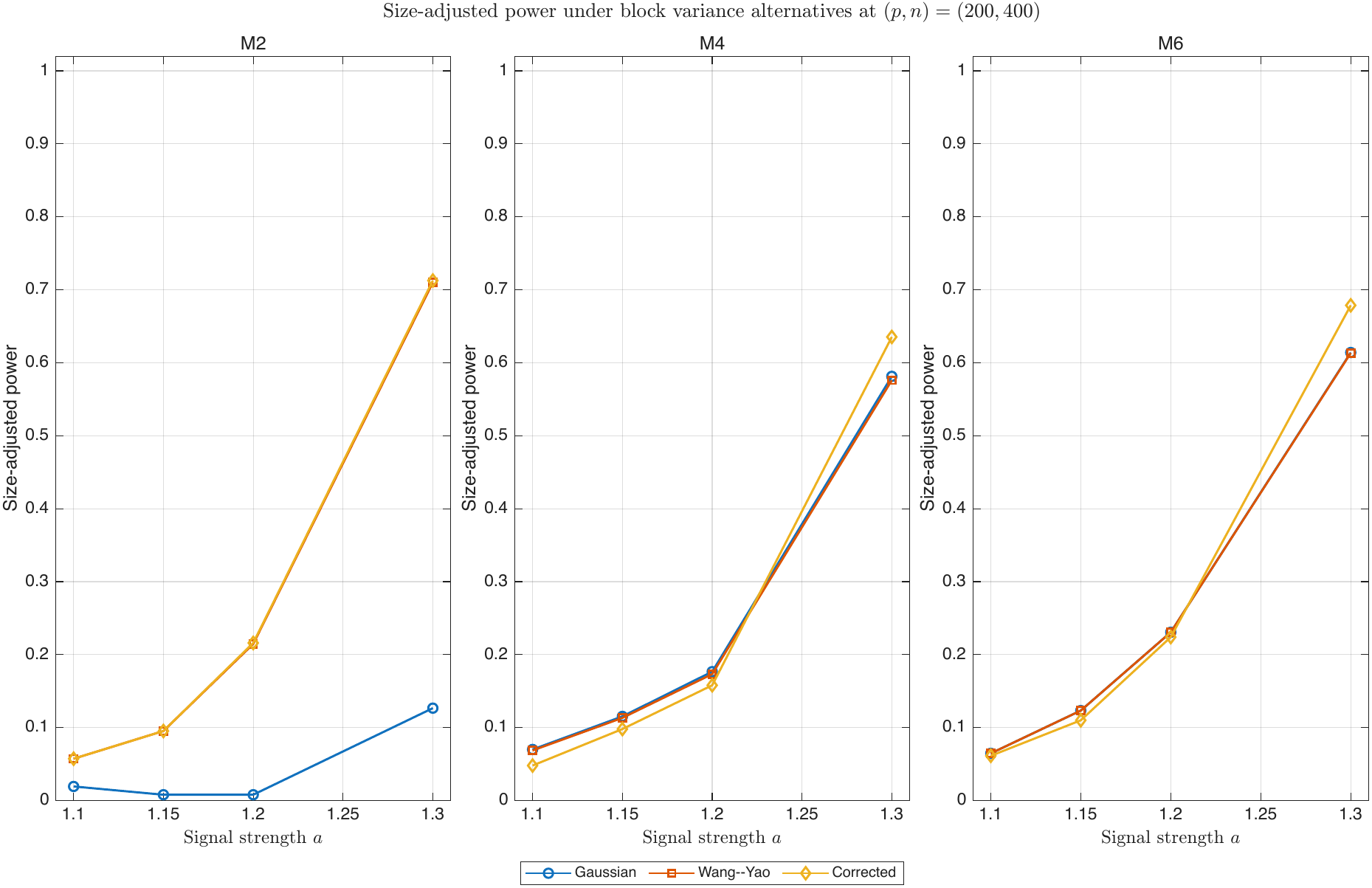}
\caption{Size-adjusted power under block variance alternatives at
$(p,n)=(200,400)$ for three representative models: the independent component benchmark (M2), a strong two-block heterogeneous model (M4), and a three-block heterogeneous model (M6).}
\label{fig:power_curves}
\end{figure}

Under M2, the proposed procedure and the Wang--Yao-type correction behave almost identically, and both dominate the Gaussian calibration. Under the heterogeneous models M4 and M6, the proposed method remains fully competitive and is slightly more powerful once the signal becomes moderately strong. Thus the correction restores null calibration without sacrificing detection ability under structured alternatives.

\subsection{Summary of the numerical evidence}

The simulations confirm the main message of the paper. In the regimes considered here, the first order spectral approximation may still be adequate, yet the classical fluctuation calibration can already fail substantially. The proposed correction succeeds because it targets the effective second order distortion identified by the GHOST framework, rather than relying on a purely Gaussian or homogeneous fourth moment approximation.

\section*{Supplementary Material}
The Supplementary Material \citet{Yin2026} contains technical proofs of the main theorems and intermediate results, together with auxiliary lemmas, model verification arguments, and calculation details omitted from the main text in order to streamline the exposition, including the residue calculation underlying Proposition~\ref{prop:john_bias_terms}.

\bibliographystyle{imsart-nameyear}

\newcommand{\gfour}{\textcircled{\scriptsize 4}}
\newcommand{\gfive}{\textcircled{\scriptsize 5}}
\newcommand{\gsix}{\textcircled{\scriptsize 6}}

\allowdisplaybreaks

\begin{frontmatter}
\title{Supplementary Material for ``The Geometry of Spectral Fluctuations: On Near-Optimal Conditions for Universal Gaussian CLTs, with Statistical Applications''}

\begin{aug}
%%%%%%%%%%%%%%%%%%%%%%%%%%%%%%%%%%%%%%%%%%%%%%%
%% Only one address is permitted per author. %%
%% Only division, organization and e-mail is %%
%% included in the address.                  %%
%% Additional information can be included in %%

\author[A]{\fnms{Yanqing} \snm{ Yin}%\thanksref{t3}
			\ead[label=e1]{yinyq799@nenu.edu.cn}}			
			\and
\author[B]{\fnms{Wang} \snm{ Zhou}%\thanksref{t4}
			\ead[label=e2]{wangzhou@nus.edu.sg}}
			
\runauthor{Yanqing Yin and Wang Zhou}
%%%%%%%%%%%%%%%%%%%%%%%%%%%%%%%%%%%%%%%%%%%%%%
%% Addresses                                %%
%%%%%%%%%%%%%%%%%%%%%%%%%%%%%%%%%%%%%%%%%%%%%%
\address[A]{School of Statistics and Data Science,
			Nanjing Audit University\printead[presep={,\ }]{e1}}

\address[B]{Department of Statistics and Data Science,  National University of Singapore \printead[presep={,\ }]{e2}}
\end{aug}

\begin{abstract}
This Supplementary Material contains technical proofs of the main results in the paper, together with auxiliary lemmas, model level verifications of the GHOST conditions, and additional calculation details omitted from the main text.
\end{abstract}

\begin{keyword}[class=MSC]
\kwd[Primary ]{62H15}
\kwd{62B20}
\kwd[; secondary ]{62D10}
\end{keyword}

\begin{keyword}
sample covariance matrix, linear spectral statistics, high-dimensionality, random matrix theory, John's test, spherical test
\end{keyword}

\end{frontmatter}

\section{Proof of Lemma \ref{lem:gamma_cumulant}}
By definition,
\begin{align*}
{\rm Cov}\!\left(Q_n(\mathbf{A}_n),Q_n(\mathbf{B}_n)\right)
&=
\mathbb{E}\Bigl[\bigl(\by_n^T\mathbf{A}_n\by_n-{\rm tr}(\mathbf{A}_n\mathbf{\Sigma}_n)\bigr)
\bigl(\by_n^T\mathbf{B}_n\by_n-{\rm tr}(\mathbf{B}_n\mathbf{\Sigma}_n)\bigr)\Bigr]\\
&=
\mathbb{E}\bigl(\by_n^T\mathbf{A}_n\by_n\,\by_n^T\mathbf{B}_n\by_n\bigr)
-{\rm tr}(\mathbf{A}_n\mathbf{\Sigma}_n)\,{\rm tr}(\mathbf{B}_n\mathbf{\Sigma}_n).
\end{align*}
Moreover,
\[
\by_n^T\mathbf{A}_n\by_n\,\by_n^T\mathbf{B}_n\by_n
=
\sum_{i,j,k,l=1}^{p}a_{ij}b_{kl}\,y_i y_j y_k y_l,
\]
and hence
\begin{equation}\label{eq:fourth_expand}
\mathbb{E}\bigl(\by_n^T\mathbf{A}_n\by_n\,\by_n^T\mathbf{B}_n\by_n\bigr)
=
\sum_{i,j,k,l=1}^{p}a_{ij}b_{kl}\,\mathbb{E}(y_i y_j y_k y_l).
\end{equation}
By the moment--cumulant decomposition, we have
\[
\mathbb{E}(y_i y_j y_k y_l)
=
[\mathbf{\Sigma}_n]_{ij}[\mathbf{\Sigma}_n]_{kl}
+[\mathbf{\Sigma}_n]_{ik}[\mathbf{\Sigma}_n]_{jl}
+[\mathbf{\Sigma}_n]_{il}[\mathbf{\Sigma}_n]_{jk}
+\kappa_{ijkl}^{(n)}.
\]
Substituting this into \eqref{eq:fourth_expand}, we obtain
\begin{align*}
\mathbb{E}\bigl(\by_n^T\mathbf{A}_n\by_n\,\by_n^T\mathbf{B}_n\by_n\bigr)
=&\sum_{i,j,k,l=1}^{p}a_{ij}b_{kl}\,[\mathbf{\Sigma}_n]_{ij}[\mathbf{\Sigma}_n]_{kl}
+\sum_{i,j,k,l=1}^{p}a_{ij}b_{kl}\,[\mathbf{\Sigma}_n]_{ik}[\mathbf{\Sigma}_n]_{jl}\\
&+\sum_{i,j,k,l=1}^{p}a_{ij}b_{kl}\,[\mathbf{\Sigma}_n]_{il}[\mathbf{\Sigma}_n]_{jk}+\sum_{i,j,k,l=1}^{p}a_{ij}b_{kl}\,\kappa_{ijkl}^{(n)}.
\end{align*}
The first term above is
\[
{\rm tr}(\mathbf{A}_n\mathbf{\Sigma}_n)\,{\rm tr}(\mathbf{B}_n\mathbf{\Sigma}_n).
\]
The second and third terms are equal to
\[
{\rm tr}(\mathbf{A}_n\mathbf{\Sigma}_n\mathbf{B}_n\mathbf{\Sigma}_n)
\qquad\text{and}\qquad
{\rm tr}(\mathbf{A}_n\mathbf{\Sigma}_n\mathbf{B}_n^T\mathbf{\Sigma}_n),
\]
respectively. Therefore,
\begin{align*}
\mathbb{E}\bigl(\by_n^T\mathbf{A}_n\by_n\,\by_n^T\mathbf{B}_n\by_n\bigr)
=
&\,{\rm tr}(\mathbf{A}_n\mathbf{\Sigma}_n)\,{\rm tr}(\mathbf{B}_n\mathbf{\Sigma}_n)\\
&+{\rm tr}(\mathbf{A}_n\mathbf{\Sigma}_n\mathbf{B}_n\mathbf{\Sigma}_n)
+{\rm tr}(\mathbf{A}_n\mathbf{\Sigma}_n\mathbf{B}_n^T\mathbf{\Sigma}_n)\\
&+\sum_{i,j,k,l=1}^{p}a_{ij}b_{kl}\,\kappa_{ijkl}^{(n)}.
\end{align*}
Subtracting ${\rm tr}(\mathbf{A}_n\mathbf{\Sigma}_n)\,{\rm tr}(\mathbf{B}_n\mathbf{\Sigma}_n)$ from both sides yields \eqref{eq:gamma_cumulant_decomp}. The bilinearity of $\Gamma_n(\cdot,\cdot)$ is immediate from its definition.

\section{Proof scheme for Theorem \ref{thm:lsd_general} and Theorem \ref{thm:lss_general_gamma}}
\label{supp:general_proof}

In this section, we explain how the abstract Conditions~\ref{cond:g1}--\ref{cond:g6}
fit into the Bai--Silverstein method. Our purpose is twofold.
First, we show that Conditions~\ref{cond:g1}--\ref{cond:g4} already suffice for the
first order spectral limit in Theorem~\ref{thm:lsd_general}. Second, we isolate the
precise points at which Conditions~\ref{cond:g4}--\ref{cond:g6} enter the fluctuation
analysis underlying Theorem~\ref{thm:lss_general_gamma}.

To make the role of the assumptions transparent, we use the markers
\[
\gfour\ \text{for Condition~\ref{cond:g4}},\qquad
\gfive\ \text{for Condition~\ref{cond:g5}},\qquad
\gsix\ \text{for Condition~\ref{cond:g6}}.
\]

We begin with the first order spectral limit. By Lemma~\ref{lem:gamma_cumulant},
for any deterministic matrix $\bA_n$,
\[
\Var\bigl(Q_n(\bA_n)\bigr)
=
\tr(\bA_n\bSig_n\bA_n\bSig_n)
+\tr(\bA_n\bSig_n\bA_n^\top\bSig_n)
+\Gamma_n(\bA_n,\bA_n).
\]
If $\|\bA_n\|=O(1)$, then the first two terms are $O(p)$ because $\|\bSig_n\|$ is
uniformly bounded by Condition~\ref{cond:g2}. For the last term, applying
Condition~\ref{cond:g4} with $\widetilde{\bA}_n=0$ and $\bB_n=\bA_n$, and using the
bilinearity of $\Gamma_n(\cdot,\cdot)$, we obtain
\[
\bigl|\Gamma_n(\bA_n,\bA_n)\bigr|
=
o\Bigl(
p^{3/2}\|\bA_n\|_{\mathrm{bil},L_1}
\Bigr).
\]
Since $\bA_n$ is deterministic, $\|\bA_n\|_{\mathrm{bil},L_1}=\|\bA_n\|=O(1)$, hence
\[
\Gamma_n(\bA_n,\bA_n)=o(p^{3/2})=o(p^2). \qquad  \hfill \gfour
\]
Therefore,
\[
\Var\bigl(Q_n(\bA_n)\bigr)=o(p^2),
\]
uniformly over deterministic $\bA_n$ with bounded spectral norm. This is exactly the
quadratic-form variance condition required in \citet{BaiZ08L}. It follows that the
empirical spectral distribution of $\bS_n$ converges weakly, almost surely, to the
M--P type limit $F^{c,H}$. This establishes Theorem~\ref{thm:lsd_general}.

We now turn to Theorem~\ref{thm:lss_general_gamma}. The proof follows the classical Bai--Silverstein contour--resolvent--martingale scheme for linear spectral statistics,
but with an additional fourth order correction tracked through the bilinear kernel
$\Gamma_n(\cdot,\cdot)$. Our purpose here is not to reproduce the classical argument
line by line. Rather, we identify the precise points at which the nonclassical inputs
enter, namely the deterministic replacement of the $\Gamma_n$ terms, the high moment bounds used in the martingale argument, and the spectral localization needed for contour
truncation and tightness. Once these inputs are available, the remaining contour--resolvent analysis
is exactly the standard Bai--Silverstein scheme. The full proof in the blockwise mixed
radial model given in the next section should therefore be viewed as a concrete realization of this abstract framework.
\subsection{Contour truncation}

Let $\mathcal I$ be the deterministic interval in \eqref{eq:confine_interval}, and let $\mathcal C$ be a positively oriented contour enclosing $\mathcal I$ inside a domain on which the test functions are analytic. Then
\begin{align}\label{eq:scheme_cauchy_aos}
L_n(f)
=
-\frac1{2\pi i}\oint_{\mathcal C} f(z)M_n(z)\,dz,
\qquad
M_n(z)=p\big(m_n(z)-m_n^0(z)\big).
\end{align}

As in Bai and Silverstein (2004), one replaces $\mathcal C$ by a truncated
contour $\mathcal C_n$ and defines a truncated process $\widehat M_n(z)$ which
agrees with $M_n(z)$ on $\mathcal C_n$ and is frozen on the removed short
segments. 

By Condition~\ref{cond:g6}, the resolvent is uniformly bounded on
$\mathcal C_n$ with overwhelming probability. Hence the contribution of the
removed segments is negligible, and it is enough to work with $\widehat M_n(z)$.
\hfill \gsix

\subsection{Martingale decomposition}

Write
\[
M_n(z)=M_{n1}(z)+M_{n2}(z),
\]
where
\[
M_{n1}(z)=p\big(m_n(z)-\E m_n(z)\big),
\qquad
M_{n2}(z)=p\big(\E m_n(z)-m_n^0(z)\big).
\]
The term $M_{n1}$ is handled by a martingale argument, while $M_{n2}$ gives the
deterministic centering correction.

Introduce the usual leave-one-out quantities
\[
\bD_k(z)=\bD(z)-\frac1n\br_k\br_k^\top,
\
\beta_k(z)=\frac1{1+\frac1n\br_k^\top\bD_k^{-1}(z)\br_k},
\
\va_k(z)=\br_k^\top\bD_k^{-1}(z)\br_k-\tr(\bD_k^{-1}(z)\bSig_n).
\]
Then the rank-one perturbation identity yields
\[
M_{n1}(z)
=
-\frac1n\sum_{k=1}^n(\E_k-\E_{k-1})
\Big[\beta_k(z)\br_k^\top\bD_k^{-2}(z)\br_k\Big].
\]

Expanding around the deterministic center
\[
b_n(z)=\Big(1+\frac1n\E\tr(\bD_1^{-1}(z)\bSig_n)\Big)^{-1},
\]
one obtains the primitive representation
\[
M_{n1}(z)=\frac{d}{dz}\widetilde M_n(z)+o_p(1),
\qquad
\widetilde M_n(z)=
-\frac{b_n(z)}{n}\sum_{k=1}^n(\E_k-\E_{k-1})\va_k(z),
\]
uniformly on $\mathcal C_n$. The remainder estimate follows from Condition
\ref{cond:g5} together with the resolvent bound from Condition \ref{cond:g6}.
\hfill \gfive\gsix

\subsection{Conditional covariance}

Let
\[
\widetilde Y_k(z)=
-\frac{b_n(z)}{n}(\E_k-\E_{k-1})\va_k(z),
\qquad
\widetilde M_n(z)=\sum_{k=1}^n\widetilde Y_k(z).
\]
For fixed $z_1,z_2\in\mathcal C_n$,
\[
\sum_{k=1}^n\E_{k-1}\big[\widetilde Y_k(z_1)\widetilde Y_k(z_2)\big]
=
\frac{b_n(z_1)b_n(z_2)}{n^2}
\sum_{k=1}^n
\E_{k-1}\big[\E_k\va_k(z_1)\E_k\va_k(z_2)\big].
\]

Lemma \ref{lem:gamma_cumulant} gives
\begin{align*}
	&\E_{k-1}\big[\E_k\va_k(z_1)\E_k\va_k(z_2)\big]\\\notag
=&[2\tr\big(\E_k\bD_k^{-1}(z_1)\bSig_n\E_k\bD_k^{-1}(z_2)\bSig_n\big)
+
\Gamma_n\big(\E_k\bD_k^{-1}(z_1),\E_k\bD_k^{-1}(z_2)\big)].
\end{align*}

Accordingly, the conditional covariance splits into a Gaussian part and a
$\Gamma_n$-part.

The Gaussian part is treated exactly as in Bai and Silverstein (2004), yielding
the logarithmic kernel
\[
-2\log\Bigg(
1-c_n\underline m(z_1)\underline m(z_2)
\int\frac{t^2}{(1+t\underline m(z_1))(1+t\underline m(z_2))}\,dH_n(t)
\Bigg).
\]

For the correction term, one uses a deterministic equivalent approximation on $\mathcal C_n$, together with Condition \ref{cond:g4}, to replace
$
\Gamma_n\big(\bD_k^{-1}(z_1),\bD_k^{-1}(z_2)\big)
$
by
$
(z_1z_2)^{-1}\Gamma_n\big(\bH_n^{-1}(z_1),\bH_n^{-1}(z_2)\big).
$  This yields the finite-$n$ covariance approximant
\[
\frac{b_n(z_1)b_n(z_2)}{n}
\Gamma_n\big(\bH_n^{-1}(z_1),\bH_n^{-1}(z_2)\big),
\]
which leads, after differentiation and contour integration, to the correction
term in Theorem~\ref{thm:lss_general_gamma}. \hfill \gfour

\subsection{Finite dimensional convergence and tightness}

Finite dimensional convergence follows from a martingale central limit theorem.
The Lindeberg bound is a consequence of Condition \ref{cond:g5}, while Condition
\ref{cond:g6} provides the required uniform resolvent control on $\mathcal C_n$.
\hfill \gfive\gsix 

Tightness of the fluctuation part follows from the resolvent identity
\[
\bD^{-1}(z_1)-\bD^{-1}(z_2)
=
(z_1-z_2)\bD^{-1}(z_1)\bD^{-1}(z_2),
\]
combined with Burkholder's inequality, Condition \ref{cond:g5}, and the
resolvent bound from Condition \ref{cond:g6}. \hfill \gfive\gsix

\subsection{Deterministic centering}

It remains to analyze
\[
M_{n2}(z)=p\bigl(\E m_n(z)-m_n^0(z)\bigr).
\]
As in the classical Bai--Silverstein argument, one first rewrites
$M_{n2}(z)$ through the difference between the exact mean resolvent
$\E\bD^{-1}(z)$ and a deterministic equivalent. Equivalently, one reduces
$M_{n2}(z)$ to a scalar error term arising from the difference between the
exact fixed point relation for $\E \underline m_n(z)$ and the deterministic
M--P equation for $\underline m_n^0(z)$.

After a second order expansion, the Gaussian trace contribution yields the
classical centering term $\mathfrak m_{0,n}(z)$. The new contribution comes
from the quadratic form covariance expansion in Lemma~\ref{lem:gamma_cumulant},
applied at the next order to quadratic forms whose matrix arguments are of resolvent type.
Because one factor arises from the first order resolvent term and the other
from the derivative term in the expansion, the correction naturally involves
one resolvent factor and one squared resolvent factor, leading to
\[
\Gamma_n\bigl(\E\bD_k^{-1}(z),\,\E\bD_k^{-2}(z)\bigr).
\]

Condition~\ref{cond:g4} then permits replacement of these random matrix arguments by their deterministic equivalents, at the scale relevant for the centering term:
\[
\Gamma_n\bigl(\E\bD_k^{-1}(z),\,\E\bD_k^{-2}(z)\bigr)
=
(-z)^{-3}\Gamma_n\bigl(\bH_n^{-1}(z),\,\bH_n^{-2}(z)\bigr)+o(1).
\qquad \hfill \gfour
\]
Consequently,
\[
\mathfrak m_n(z)=\mathfrak m_{0,n}(z)+\mathfrak m_{1,n}(z),
\]
where $\mathfrak m_{0,n}(z)$ is the classical Bai--Silverstein centering term
and $\mathfrak m_{1,n}(z)$ is the fourth order correction generated by
$\Gamma_n(\bH_n^{-1}(z),\bH_n^{-2}(z))$. Contour integration then yields the
mean approximant $\mathcal M_n(f)$ in Theorem~\ref{thm:lss_general_gamma}. \hfill \gfour

\subsection{Relation to the blockwise proof}

The blockwise proof should be viewed as a concrete realization of the abstract framework above. Once the first order spectral limit is the same M--P type law, the contour representation, the leave-one-out resolvent expansion, the martingale decomposition, the treatment of the Gaussian trace term, and the analytic manipulations based on the fixed point equation are all model free. These steps depend only on the resolvent structure of sample covariance matrices and on the deterministic equivalent associated with $F^{c,H}$.

The genuinely model dependent part is confined to a small number of inputs: the concentration of quadratic forms, the high moment bounds needed for the martingale central limit theorem and tightness, the spectral localization required for contour truncation, and the deterministic replacement of the fourth order correction terms involving $\Gamma_n$. Conditions~\ref{cond:g4}--\ref{cond:g6} are formulated precisely to isolate these inputs. The proof of Theorem~\ref{thm2} verifies them in the blockwise mixed radial model. Once these verifications are available, the remaining argument follows exactly the general scheme described above.

\section{Proof of Theorem \ref{thm2}}
First, for any deterministic matrix $\bA$ with uniformly bounded spectral norm,
apply Lemma~\ref{mlez2} to $\bga_n^\top \bA \bga_n$ to obtain
\begin{align*}
\E\left|\br_1^\top \bA \br_1-\tr(\bA\bSig_n)\right|^2
&=
\E\left|\bw_1^\top (\bga_n^\top\bA\bga_n)\bw_1-\tr(\bga_n^\top\bA\bga_n)\right|^2\\
&\le C\bigl[\eta_n^2pn+p^{2\alpha_1-\delta_1}+p^{2\alpha_2-\delta_2}\bigr]
=o(p^2).
\end{align*}
In particular, for any deterministic sequence of matrices $\bA=\bA_p$ with uniformly bounded spectral norm,
\[
\Var\big(\br_1^\top \bA \br_1\big)=o(p^2).
\]
This is exactly the quadratic-form concentration condition required in
\cite{BaiZ08L}. Therefore, by the main results of \cite{BaiZ08L}, the empirical
spectral distribution of the sample covariance matrix under the block model
still converges almost surely to the M--P type limit
$F^{c,H}$.

We now carry out the abstract contour integration, resolvent expansion, and martingale decomposition scheme described in
Section~\ref{supp:general_proof} in the concrete blockwise mixed radial model.
We keep a full model level proof here because, in the present model, the natural
fluctuation normalization may differ from the classical one, and the fourth order
correction kernel can be computed explicitly. Thus the argument below should be
viewed as a full realization of the abstract scheme rather than as a separate ad hoc proof.

As in \cite{Bai2004a,Hu2019}, we first truncate and recenter the underlying variables. Let
$\eta_n\downarrow0$ sufficiently slowly. It is enough to prove Theorem
\ref{thm2} under the truncated setting
\begin{align}\label{mal2t}
|x_{jk}|\le \eta_n\sqrt n,\qquad
|y_{tk}|\le \eta_n\sqrt n,\qquad
|\rho_{xk}^2-1|\le \eta_n,\qquad
|\rho_{yk}^2-1|\le \eta_n,
\end{align}
for all relevant indices $j,t,k$. We therefore work throughout under
\eqref{mal2t}.

Recall that
\[
M_n(z)=pr_p(\kappa)\bigl(m_n(z)-m_n^0(z)\bigr),
\]
where $m_n(z)=m_{F^{\bS_n}}(z)$ and $m_n^0(z)=m_{F^{c_n,H_n}}(z)$. As in the
classical proof of Bai and Silverstein (2004), the proof proceeds by analyzing
the process $M_n(z)$ on a suitable truncated contour.

Let $v_0>0$, and choose numbers $x_r,\eta_r$ such that
\[
x_r>\eta_r>\limsup_n\lambda_{\max}^{\bSig_n}(1+\sqrt c)^2.
\]
If the left endpoint of the interval in \eqref{int} is equal to $0$, choose
negative numbers $x_l<\eta_l<0$. Otherwise choose
\[
0<x_l<\eta_l<
\liminf_n \lambda_{\min}^{\bSig_n}I_{(0,1)}(c)(1-\sqrt c)^2.
\]
Define
\[
\mathcal C_u=\{x+iv_0:\ x\in[x_l,x_r]\},
\]
and let
\[
\mathcal C=
\{x_l+iv:\ v\in[0,v_0]\}
\cup
\mathcal C_u
\cup
\{x_r+iv:\ v\in[0,v_0]\}.
\]
This contour encloses the support of the deterministic equivalent
$F^{c_n,H_n}$ for all sufficiently large $n$.

To avoid the singular behavior near the real axis, we truncate the two vertical
segments. Choose a sequence $\{\va_n\}$ such that
\[
\va_n\downarrow0,
\qquad
\va_n\ge n^{-\alpha}
\]
for some fixed $\alpha\in(0,1)$. Set
\[
\mathcal C_l=\{x_l+iv:\ v\in[n^{-1}\va_n,v_0]\},
\qquad
\mathcal C_r=\{x_r+iv:\ v\in[n^{-1}\va_n,v_0]\},
\]
and define the truncated contour
\[
\mathcal C_n=\mathcal C_l\cup\mathcal C_u\cup\mathcal C_r.
\]
We also define the truncated process $\widehat M_n(z)$ by
\begin{align}\label{aa}
\widehat M_n(z)=
\begin{cases}
M_n(z), & z\in\mathcal C_n,\\
M_n(x_l+in^{-1}\va_n), & z=x_l+iv,\ v\in[0,n^{-1}\va_n],\\
M_n(x_r+in^{-1}\va_n), & z=x_r+iv,\ v\in[0,n^{-1}\va_n].
\end{cases}
\end{align}

The central limit theorem for $\widehat M_n(z)$ is stated below.
For $z\in\mathcal C\cup\overline{\mathcal C}$, define
\[
\bH(z)=\underline m(z)\bSig_n+\bI_p.
\]

\begin{lemma}\label{th2}
Under the conditions of Theorem \ref{thm2}, the process $\widehat M_n(z)$
converges weakly on $\mathcal C$ to a Gaussian process $M(z)$ satisfying, for
$z\in\mathcal C$,
\begin{align}\label{mvar}
\begin{split}
\E M(z)
=&\
\frac{cr_p(\kappa)\displaystyle\int\frac{\underline m^3(z)x^2}
{\left(1+\underline m(z)x\right)^3}\,dH(x)}
{\left(1-c\displaystyle\int\frac{\underline m^2(z)x^2}
{\left(1+\underline m(z)x\right)^2}\,dH(x)\right)^2} \\
&\ +\frac{\underline m^3(z)}
{1-c\displaystyle\int\frac{\underline m^2(z)x^2}
{\left(1+\underline m(z)x\right)^2}\,dH(x)}
\cdot \lim_{n\to\infty}\frac{r_p(\kappa)}{n}
\Gamma_n^{\mathrm{blk}}\big(\bH^{-1}(z),\bH^{-2}(z)\big).
\end{split}
\end{align}
and, for $z_1,z_2\in\mathcal C\cup\overline{\mathcal C}$,
\begin{align*}
\Cov\bigl(M(z_1),M(z_2)\bigr)
=&\
2r_p^2(\kappa)\left(
\frac{\underline m'(z_1)\underline m'(z_2)}
{(\underline m(z_1)-\underline m(z_2))^2}
-\frac1{(z_1-z_2)^2}
\right) \\
&\ +\underline m'(z_1)\underline m'(z_2)
\lim_{n\to\infty}\frac{r_p^2(\kappa)}{n}
\Gamma_n^{\mathrm{blk}}\big(\bH^{-1}(z_1),\bH^{-1}(z_2)\big).
\end{align*}
Here $\Gamma_n^{\mathrm{blk}}$ is given by \eqref{eq:block_gamma_thm}. In particular,
\begin{align*}
\Gamma_n^{\mathrm{blk}}\big(\bH^{-1}(z),\bH^{-2}(z)\big)
=
&\ \nu_{x4}\tr\Big[\big(\bga_n^\top\bH^{-1}(z)\bga_n\circ\bga_n^\top\bH^{-2}(z)\bga_n\big)\bF^1_{p_1}\Big]\\
&+\nu_{y4}\tr\Big[\big(\bga_n^\top\bH^{-1}(z)\bga_n\circ\bga_n^\top\bH^{-2}(z)\bga_n\big)\bF^2_{p_2}\Big]\\
&+\tau_x p^{-\delta_1}\tr\big(\bga_n^\top\bH^{-1}(z)\bga_n\bF_{p_1}^1\big)
\tr\big(\bga_n^\top\bH^{-2}(z)\bga_n\bF_{p_1}^1\big)\\
&+\tau_y p^{-\delta_2}\tr\big(\bga_n^\top\bH^{-1}(z)\bga_n\bF_{p_2}^2\big)
\tr\big(\bga_n^\top\bH^{-2}(z)\bga_n\bF_{p_2}^2\big),
\end{align*}
and similarly
\begin{align*}
\Gamma_n^{\mathrm{blk}}\big(\bH^{-1}(z_1),\bH^{-1}(z_2)\big)
=
&\ \nu_{x4}\tr\Big[\big(\bga_n^\top\bH^{-1}(z_1)\bga_n\circ\bga_n^\top\bH^{-1}(z_2)\bga_n\big)\bF^1_{p_1}\Big]\\
&+\nu_{y4}\tr\Big[\big(\bga_n^\top\bH^{-1}(z_1)\bga_n\circ\bga_n^\top\bH^{-1}(z_2)\bga_n\big)\bF^2_{p_2}\Big]\\
&+\tau_x p^{-\delta_1}\tr\big(\bga_n^\top\bH^{-1}(z_1)\bga_n\bF_{p_1}^1\big)
\tr\big(\bga_n^\top\bH^{-1}(z_2)\bga_n\bF_{p_1}^1\big)\\
&+\tau_y p^{-\delta_2}\tr\big(\bga_n^\top\bH^{-1}(z_1)\bga_n\bF_{p_2}^2\big)
\tr\big(\bga_n^\top\bH^{-1}(z_2)\bga_n\bF_{p_2}^2\big).
\end{align*}
\end{lemma}

We now explain why Lemma \ref{th2} implies Theorem \ref{thm2}. Since
$F^{\bS_n}\Rightarrow F^{c,H}$ almost surely, the support of $F^{c_n,H_n}$ is
contained in the interval \eqref{int} for all sufficiently large $n$. Hence,
for every $f\in\{f_1,\dots,f_s\}$, Cauchy's integral formula gives
\[
\int f(x)\,dG_n(x)
=
-\frac{1}{2\pi i}\oint_{\mathcal C\cup\overline{\mathcal C}}
f(z)M_n(z)\,dz
\]
with probability tending to one.

Moreover, by construction of $\widehat M_n(z)$ and by the resolvent bounds on
the removed segments,
\[
\left|
\oint_{\mathcal C\cup\overline{\mathcal C}}
f(z)\bigl(M_n(z)-\widehat M_n(z)\bigr)\,dz
\right|
\to 0
\]
in probability. Therefore,
\[
\int f(x)\,dG_n(x)
=
-\frac{1}{2\pi i}\oint_{\mathcal C\cup\overline{\mathcal C}}
f(z)\widehat M_n(z)\,dz
+o_p(1).
\]
Since the map
\[
\widehat M_n(\cdot)\mapsto
\left(
-\frac1{2\pi i}\oint f_1(z)\widehat M_n(z)\,dz,\,
\dots,\,
-\frac1{2\pi i}\oint f_s(z)\widehat M_n(z)\,dz
\right)
\]
is continuous, Theorem \ref{thm2} follows from Lemma \ref{th2} by the continuous
mapping theorem.

It remains to prove Lemma \ref{th2}. For $z\in\mathcal C_n$, write
\[
M_n(z)=M_{n1}(z)+M_{n2}(z),
\]
where
\[
M_{n1}(z)=pr_p(\kappa)\bigl(m_n(z)-\E m_n(z)\bigr),
\qquad
M_{n2}(z)=pr_p(\kappa)\bigl(\E m_n(z)-m_n^0(z)\bigr).
\]
The proof is divided into three parts:

\begin{itemize}
\item[(i)] finite dimensional convergence of the fluctuation part $M_{n1}(z)$;

\item[(ii)] tightness of $M_{n1}(z)$ on the truncated contour $\mathcal C_n$;

\item[(iii)] convergence of the deterministic centering part $M_{n2}(z)$.
\end{itemize}

\subsection{Preliminary resolvent relations}

We first record several preliminary resolvent relations that will be used
repeatedly in the proof below. Although the LSD under the block model follows
from \cite{BaiZ08L}, we include the short derivation here because the
martingale decomposition of $m_n(z)$, the control of $\beta_k(z)-b(z)$, and
the relation between $\E\beta_k(z)$ and $\E\underline m_n(z)$ will all be
needed later in the analysis of both $M_{n1}(z)$ and $M_{n2}(z)$.

Using the formula
\begin{align*}
\bigl(\bA+\mb q\mb q^\top\bigr)^{-1}
=
\bA^{-1}-\frac{\bA^{-1}\mb q\mb q^\top\bA^{-1}}{1+\mb q^\top\bA^{-1}\mb q},
\end{align*}
it follows that
\begin{align}\label{mclt}
\begin{split}
m_n(z)-\E m_n(z)
&=\frac1p\sum_{k=1}^n(\E_k-\E_{k-1})
\tr\bigl(\bD^{-1}(z)-\bD_k^{-1}(z)\bigr)\\
&=-\frac1{pn}\sum_{k=1}^n(\E_k-\E_{k-1})
\beta_k(z)\br_k^\top\bD_k^{-2}(z)\br_k.
\end{split}
\end{align}
By Lemma \ref{burk}, we have
\begin{align*}
\E\left|m_n(z)-\E m_n(z)\right|^4\le \frac{C}{n^2}.
\end{align*}
Together with the Borel--Cantelli lemma, this implies
\begin{align}\label{mal6}
m_n(z)-\E m_n(z)\xrightarrow{a.s.}0.
\end{align}

Define
\begin{align*}
{\bf T}(z)=b(z)\bSig_n-z\bI_p.
\end{align*}
Then
\begin{align*}
\bD^{-1}(z)-{\bf T}^{-1}(z)
=
-\frac1n\sum_{k=1}^n
{\bf T}^{-1}(z)\bigl(\br_k\br_k^\top-b(z)\bSig_n\bigr)\bD^{-1}(z).
\end{align*}
Hence,
\begin{align*}
&\E m_n(z)-\frac1p\tr\bigl({\bf T}^{-1}(z)\bigr)\\
=&-\frac1{pn}\sum_{k=1}^n\E\bigg[
\beta_k(z)\Big(
\br_k^\top\bD_k^{-1}(z){\bf T}^{-1}(z)\br_k
-\tr\bigl(\bSig_n\bD_k^{-1}(z){\bf T}^{-1}(z)\bigr)
\Big)\bigg]\\
&-\frac1{pn}\sum_{k=1}^n\E\bigg[
\bigl(\beta_k(z)-b(z)\bigr)
\tr\bigl({\bf T}^{-1}(z)\bSig_n\bD_k^{-1}(z)\bigr)
\bigg]\\
&-\frac{b(z)}{pn}\sum_{k=1}^n
\E\tr\bigl({\bf T}^{-1}(z)\bigl(\bD_k^{-1}(z)-\bD^{-1}(z)\bigr)\bSig_n\bigr).
\end{align*}
Using Lemma \ref{mle2}, we obtain
\begin{align*}
\left|
\E m_n(z)-\frac1p\tr\bigl({\bf T}^{-1}(z)\bigr)
\right|
\le&
\frac{C}{pn}\sum_{k=1}^n
\E^{1/2}\left|
\br_k^\top\bD_k^{-1}(z){\bf T}^{-1}(z)\br_k
-\tr\bigl(\bSig_n\bD_k^{-1}(z){\bf T}^{-1}(z)\bigr)
\right|^2\\
&+\frac{C}{n}\sum_{k=1}^n\E^{1/2}\left|\beta_k(z)-b(z)\right|^2
+\frac{C}{p}\to0,
\end{align*}
where we used
\begin{align}\label{mal3}
\begin{split}
\E\left|\beta_k(z)-b(z)\right|^2
\le&
\frac C{n^2}\E\left|
\br_k^\top\bD_k^{-1}(z)\br_k-\tr\bD_k^{-1}(z)\bSig_n
\right|^2\\
&+\frac C{n^2}\E\left|
\tr\bD_k^{-1}(z)\bSig_n-\E\tr\bD_k^{-1}(z)\bSig_n
\right|^2\\
\le&\ o(1)+\frac Cn\to0.
\end{split}
\end{align}
Therefore,
\begin{align}\label{mal4}
\E m_n(z)-\int\frac1{b(z)t-z}\,dH_n(t)\to0.
\end{align}

Now write
\begin{align*}
\bD(z)+z\bI_p=\frac1n\sum_{k=1}^n\br_k\br_k^\top.
\end{align*}
Multiplying both sides by $\bD^{-1}(z)$ on the right, taking traces, and then
taking expectations, we get
\begin{align*}
1+z\E m_n(z)
=
\frac1{np}\sum_{k=1}^n\E\bigl(\beta_k(z)\br_k^\top\bD_k^{-1}(z)\br_k\bigr)
=
\frac np-\frac np\E\beta_1(z).
\end{align*}
Together with \begin{align}\label{mal1}
	\underline m_n(z)=c_n m_n(z)-\frac{1-c_n}z,
\end{align} this yields
\begin{align}\label{mal7}
\E\beta_1(z)=-z\,\E\underline m_n(z).
\end{align}
By \eqref{mal1}, \eqref{mal3}, \eqref{mal4}, and \eqref{mal7}, it follows that
\begin{align}\label{mal5}
\E m_n(z)-\int\frac1{(1-c_n-c_nz\E m_n(z))t-z}\,dH_n(t)\to0.
\end{align}

The relations \eqref{mclt}, \eqref{mal3}, \eqref{mal4}, \eqref{mal5}, and
\eqref{mal7} will be used repeatedly in the sequel.

\subsection{The limiting distribution of $M_{n1}(z)$}

We first establish the finite dimensional convergence of the fluctuation part
$M_{n1}(z)$. Throughout this subsection, $z,z_1,z_2$ are fixed points in
$\mathcal C_n$.

Let
\[
\bD(z)=\bS_n-z\bI_p,
\qquad
\bD_k(z)=\bD(z)-\frac1n\br_k\br_k^\top,
\]
and define
\[
\beta_k(z)=\frac{1}{1+\frac1n\br_k^\top\bD_k^{-1}(z)\br_k},
\qquad
\bar\beta_k(z)=\frac{1}{1+\frac1n\tr(\bD_k^{-1}(z)\bSig_n)}.
\]
Also write
\[
\va_k(z)=\br_k^\top\bD_k^{-1}(z)\br_k-\tr(\bD_k^{-1}(z)\bSig_n),
\]
and
\[
\gamma_k(z)=\br_k^\top\bD_k^{-2}(z)\br_k-\tr(\bD_k^{-2}(z)\bSig_n).
\]

By the rank-one perturbation identity,
\[
\tr\bD^{-1}(z)-\tr\bD_k^{-1}(z)
=
-\frac1n\beta_k(z)\br_k^\top\bD_k^{-2}(z)\br_k.
\]
Hence
\[
M_{n1}(z)
=
pr_p(\kappa)\bigl(m_n(z)-\E m_n(z)\bigr)
=
-\frac{r_p(\kappa)}{n}\sum_{k=1}^n(\E_k-\E_{k-1})
\bigl[\beta_k(z)\br_k^\top\bD_k^{-2}(z)\br_k\bigr].
\]
Using
\[
\br_k^\top\bD_k^{-2}(z)\br_k
=
\gamma_k(z)+\tr(\bD_k^{-2}(z)\bSig_n),
\]
and noting that $\bD_k(z)$ is $\mathcal F_{k-1}$-measurable, we obtain
\[
M_{n1}(z)
=
-\frac{r_p(\kappa)}{n}\sum_{k=1}^n(\E_k-\E_{k-1})
\bigl[\beta_k(z)\gamma_k(z)\bigr].
\]

Next, from
\[
\beta_k(z)
=
\bar\beta_k(z)-\frac1n\beta_k(z)\bar\beta_k(z)\va_k(z),
\]
we further expand
\[
\beta_k(z)
=
\bar\beta_k(z)-\frac1n\bar\beta_k^2(z)\va_k(z)
+\frac1{n^2}\beta_k(z)\bar\beta_k^2(z)\va_k^2(z).
\]
Substituting this into the above representation yields
\begin{align*}
M_{n1}(z)
=&-\frac{r_p(\kappa)}{n}\sum_{k=1}^n(\E_k-\E_{k-1})
\bigl[\bar\beta_k(z)\gamma_k(z)\bigr]\\
&+\frac{r_p(\kappa)}{n^2}\sum_{k=1}^n(\E_k-\E_{k-1})
\bigl[\bar\beta_k^2(z)\va_k(z)\gamma_k(z)\bigr]\\
&-\frac{r_p(\kappa)}{n^3}\sum_{k=1}^n(\E_k-\E_{k-1})
\bigl[\beta_k(z)\bar\beta_k^2(z)\va_k^2(z)\gamma_k(z)\bigr].
\end{align*}

We now isolate the leading term. Since
\[
\bar\beta_k'(z)=\frac1n\bar\beta_k^2(z)\tr(\bD_k^{-2}(z)\bSig_n),
\]
it follows that
\[
\bar\beta_k(z)\gamma_k(z)-\frac1n\bar\beta_k^2(z)\va_k(z)\tr(\bD_k^{-2}(z)\bSig_n)
=
\bar\beta_k(z)\gamma_k(z)-\bar\beta_k'(z)\va_k(z).
\]
On the other hand, differentiating
\[
\bar\beta_k(z)=\frac{1}{1+\frac1n\tr(\bD_k^{-1}(z)\bSig_n)}
\]
gives
\[
\frac{d}{dz}\bigl(\bar\beta_k(z)\va_k(z)\bigr)
=
\bar\beta_k'(z)\va_k(z)+\bar\beta_k(z)\gamma_k(z).
\]
Hence
\[
\bar\beta_k(z)\gamma_k(z)-\bar\beta_k'(z)\va_k(z)
=
\frac{d}{dz}\bigl(\bar\beta_k(z)\va_k(z)\bigr)-2\bar\beta_k'(z)\va_k(z).
\]
For our purpose it is more convenient to write the leading term in the form
\[
Y_k(z):=
-\frac{r_p(\kappa)}{n}(\E_k-\E_{k-1})
\bigl[\bar\beta_k(z)\gamma_k(z)-\bar\beta_k'(z)\va_k(z)\bigr].
\]
Equivalently,
\[
Y_k(z)
=
-\frac{r_p(\kappa)}{n}(\E_k-\E_{k-1})
\bigl[\bar\beta_k(z)\gamma_k(z)\bigr]
+\frac{r_p(\kappa)}{n^2}(\E_k-\E_{k-1})
\bigl[\bar\beta_k^2(z)\va_k(z)\tr(\bD_k^{-2}(z)\bSig_n)\bigr].
\]

Therefore,
\[
M_{n1}(z)=\sum_{k=1}^nY_k(z)+o_p(1),
\]
provided the remaining terms are negligible. We verify this next.

By Lemma \ref{mlez}, together with the uniform resolvent bounds on
$\mathcal C_n$, we have
\[
\E|\va_k(z)|^4\le C\eta_n^2\bigl(n+n^{2\kappa}\bigr),
\qquad
\E|\gamma_k(z)|^4\le C\eta_n^2\bigl(n+n^{2\kappa}\bigr),
\]
uniformly in $k$ and $z\in\mathcal C_n$. It follows that
\begin{align*}
&\E\left|
\frac{r_p(\kappa)}{n^2}\sum_{k=1}^n(\E_k-\E_{k-1})
\bigl[\bar\beta_k^2(z)\va_k(z)\gamma_k(z)\bigr]
\right|^2\\
\le&
\frac{Cr_p^2(\kappa)}{n^4}\sum_{k=1}^n
\E|\va_k(z)|^2|\gamma_k(z)|^2\\
\le&
\frac{Cr_p^2(\kappa)}{n^4}\sum_{k=1}^n
\E^{1/2}|\va_k(z)|^4\,\E^{1/2}|\gamma_k(z)|^4
\le C\eta_n^2\to0.
\end{align*}
Similarly,
\begin{align*}
&\E\left|
\frac{r_p(\kappa)}{n^3}\sum_{k=1}^n(\E_k-\E_{k-1})
\bigl[\beta_k(z)\bar\beta_k^2(z)\va_k^2(z)\gamma_k(z)\bigr]
\right|^2\\
\le&
\frac{Cr_p^2(\kappa)}{n^6}\sum_{k=1}^n
\E|\va_k(z)|^4|\gamma_k(z)|^2
\le C\eta_n^2\to0.
\end{align*}
Consequently,
\[
M_{n1}(z)=\sum_{k=1}^nY_k(z)+o_p(1).
\]

We now apply the martingale central limit theorem. For any fixed $r\ge1$ and
real constants $a_1,\dots,a_r$, consider
\[
\sum_{j=1}^ra_jM_{n1}(z_j)
=
\sum_{k=1}^n\sum_{j=1}^ra_jY_k(z_j)+o_p(1).
\]
Thus it is enough to study the martingale difference array
\[
Z_{n,k}:=\sum_{j=1}^ra_jY_k(z_j).
\]

We first verify the Lindeberg condition. By Holder's inequality and the fourth
moment bounds above,
\begin{align*}
\sum_{k=1}^n\E\bigl(|Z_{n,k}|^2I(|Z_{n,k}|\ge\va)\bigr)
&\le \frac1{\va^2}\sum_{k=1}^n\E|Z_{n,k}|^4\\
&\le \frac{C}{\va^2}\sum_{k=1}^n\sum_{j=1}^ra_j^4\E|Y_k(z_j)|^4\\
&\le \frac{Cr_p^4(\kappa)}{\va^2n^4}
\sum_{k=1}^n\sum_{j=1}^ra_j^4
\Bigl(\E|\va_k(z_j)|^4+\E|\gamma_k(z_j)|^4\Bigr)\\
&\le C\eta_n^2\to0.
\end{align*}
Hence the Lindeberg condition holds.

It remains to compute the limiting conditional covariance. For fixed
$z_1,z_2\in\mathcal C_n$,
\begin{align}\label{eq:block_cov_start}
\sum_{k=1}^n\E_{k-1}[Y_k(z_1)Y_k(z_2)]
=
\frac{\partial^2}{\partial z_2\partial z_1}
\frac{r_p^2(\kappa)}{n^2}
\sum_{k=1}^n
\E_{k-1}\bigl[\E_k\bar\beta_k(z_1)\va_k(z_1)\,\E_k\bar\beta_k(z_2)\va_k(z_2)\bigr]
+o_p(1).
\end{align}
Since $\bar\beta_k(z)$ is $\mathcal F_{k-1}$-measurable, we may write
\begin{align*}
&\frac{r_p^2(\kappa)}{n^2}
\sum_{k=1}^n
\E_{k-1}\bigl[\E_k\bar\beta_k(z_1)\va_k(z_1)\,\E_k\bar\beta_k(z_2)\va_k(z_2)\bigr]\\
=&\frac{r_p^2(\kappa)}{n^2}
\sum_{k=1}^n
\bar\beta_k(z_1)\bar\beta_k(z_2)
\E_{k-1}\bigl[\E_k\va_k(z_1)\E_k\va_k(z_2)\bigr].
\end{align*}

Next we replace $\bar\beta_k(z)$ by its deterministic equivalent
\[
b(z)=\frac{1}{1+\int \frac{t}{-z(1+t\underline m(z))}\,dH(t)},
\]
or equivalently $b(z)+z\underline m(z)=0$. By Lemma \ref{mle2},
\[
\E|\bar\beta_k(z)-b(z)|^2\le \frac{C}{n}.
\]
Therefore,
\begin{align*}
&\frac{r_p^2(\kappa)}{n^2}
\sum_{k=1}^n
\bar\beta_k(z_1)\bar\beta_k(z_2)
\E_{k-1}\bigl[\E_k\va_k(z_1)\E_k\va_k(z_2)\bigr]\\
=&\frac{r_p^2(\kappa)}{n^2}b(z_1)b(z_2)
\sum_{k=1}^n
\E_{k-1}\bigl[\E_k\va_k(z_1)\E_k\va_k(z_2)\bigr]
+o_p(1).
\end{align*}

Applying Lemma \ref{mlez2}, we have
\begin{align*}
&\frac{r_p^2(\kappa)}{n^2}b(z_1)b(z_2)
\sum_{k=1}^n
\E_{k-1}\Big[
\E_k\va_k(z_1)\E_k\va_k(z_2)
\Big]\\
=&\
\frac{r_p^2(\kappa)}{n^2}b(z_1)b(z_2)\sum_{k=1}^n\Bigg[
\nu_{x4}\tr\left[
\bigl(\bga_n^\top\E_k\bD_k^{-1}(z_1)\bga_n
\circ
\bga_n^\top\E_k\bD_k^{-1}(z_2)\bga_n\bigr)\bF^1_{p_1}
\right]\\
&\qquad\qquad\qquad\qquad
+\nu_{y4}\tr\left[
\bigl(\bga_n^\top\E_k\bD_k^{-1}(z_1)\bga_n
\circ
\bga_n^\top\E_k\bD_k^{-1}(z_2)\bga_n\bigr)\bF^2_{p_2}
\right]\\
&\qquad\qquad\qquad\qquad
+2\tr\bigl(\E_k\bD_k^{-1}(z_1)\bSig_n\E_k\bD_k^{-1}(z_2)\bSig_n\bigr)\\
&\qquad\qquad\qquad\qquad
+\tau_xp^{-\delta_1}
\tr\bigl(\bga_n^\top\E_k\bD_k^{-1}(z_1)\bga_n\bF_{p_1}^1\bigr)
\tr\bigl(\bga_n^\top\E_k\bD_k^{-1}(z_2)\bga_n\bF_{p_1}^1\bigr)\\
&\qquad\qquad\qquad\qquad
+\tau_yp^{-\delta_2}
\tr\bigl(\bga_n^\top\E_k\bD_k^{-1}(z_1)\bga_n\bF_{p_2}^2\bigr)
\tr\bigl(\bga_n^\top\E_k\bD_k^{-1}(z_2)\bga_n\bF_{p_2}^2\bigr)
\Bigg]
+o_p(1)\\
=&\
\frac{r_p^2(\kappa)}{n}b(z_1)b(z_2)
\Gamma_n^{\mathrm{blk}}\bigl(\bH^{-1}(z_1),\bH^{-1}(z_2)\bigr)\\
&+\frac{2r_p^2(\kappa)}{n^2}\sum_{k=1}^n
\frac{
b(z_1)b(z_2)\tr\bigl(\bH^{-1}(z_1)\bSig_n\bH^{-1}(z_2)\bSig_n\bigr)
}{
1-\frac{k-1}{n^2}
b(z_1)b(z_2)\tr\bigl(\bH^{-1}(z_1)\bSig_n\bH^{-1}(z_2)\bSig_n\bigr)
}
+o_p(1).
\end{align*}
Here in the second equality, the logarithmic term comes from the Gaussian part
\[
2\tr\bigl(\E_k\bD_k^{-1}(z_1)\bSig_n\E_k\bD_k^{-1}(z_2)\bSig_n\bigr),
\]
and its derivation is exactly the same as in the classical argument of
Bai and Silverstein \cite{Bai2004a,BaiS10S}.
The other terms follow from the entrywise deterministic-equivalent approximation for the leave-one-out resolvent established in
\cite[Lemma 5.4]{Yin2022a}, which yields
\[
\big\|
(\bD_k^{-1}(z))
-((-z)^{-1}\bH^{-1}(z))
\big\|_{\mathrm{bil},L_1}=O(n^{-1/2}).
\]
By this bound and the explicit form of $\Gamma_n^{\mathrm{blk}}$, together with the boundedness of $\bga_n$ and the block trace structure, the induced errors in both the Hadamard trace terms and the block trace product terms are asymptotically negligible.

By \eqref{mal3} and \eqref{mal7}, we have
\begin{align*}
|b(z)+z\underline m(z)|
\le&
|\E\beta_1(z)-b(z)|+|z(\E\underline m_n(z)-\underline m(z))|\\
\le&
\E^{1/2}|\beta_1(z)-b(z)|^2+o(1)
=
o(1).
\end{align*}
Therefore,
\begin{align*}
&\frac{r_p^2(\kappa)}{n^2}b(z_1)b(z_2)
\sum_{k=1}^n
\E_{k-1}\Big[
\E_k\va_k(z_1)\E_k\va_k(z_2)
\Big]\\
=&\
\frac{r_p^2(\kappa)}{n}\underline m(z_1)\underline m(z_2)
\Gamma_n^{\mathrm{blk}}\bigl(\bH^{-1}(z_1),\bH^{-1}(z_2)\bigr)\\
&-2r_p^2(\kappa)\log\Bigg(
1-c_n\underline m(z_1)\underline m(z_2)
\int\frac{t^2}{(1+t\underline m(z_1))(1+t\underline m(z_2))}\,dH_n(t)
\Bigg)
+o_p(1).
\end{align*}
This implies
\begin{align*}
\sum_{k=1}^n\E_{k-1}[Y_k(z_1)Y_k(z_2)]
=&\
\frac{r_p^2(\kappa)}{n}\underline m'(z_1)\underline m'(z_2)
\Gamma_n^{\mathrm{blk}}\bigl(\bH^{-1}(z_1),\bH^{-1}(z_2)\bigr)\\
&+2r_p^2(\kappa)\left(
\frac{\underline m'(z_1)\underline m'(z_2)}
{(\underline m(z_1)-\underline m(z_2))^2}
-\frac1{(z_1-z_2)^2}
\right)
+o_p(1).
\end{align*}

\subsection{Tightness of $M_{n1}(z)$}

Using Theorem 12.3 of Billingsley \cite{Billingsley09C}, it
suffices to show
\begin{align*}
\sup_{n,z_1,z_2\in\mathcal{C}_n}
\frac{\E|M_{n1}(z_1)-M_{n1}(z_2)|^2}{|z_1-z_2|^2}\le C.
\end{align*}
We have that, for $z\in\mathcal C_n$,
\begin{align*}
\|\bD^{-1}(z)\|
\le
\frac1{x_r-\eta_r}+\frac1{\eta_l-x_l}+\frac1{v_0}
+n\va_n^{-1}
I\bigl(\|\bS_n\|\ge\eta_r\ {\rm or}\ \lambda_{\min}^{\bS_n}<\eta_l\bigr).
\end{align*}
This implies that, for $\ell\ge1$,
\[
\E\|\bD^{-1}(z)\|^{\ell}\le C.
\]
We also claim that the moments of $\|\bD_k^{-1}(z)\|$ are uniformly bounded in
$n$ and $z\in\mathcal C_n$. Since
\[
\frac1n\br_k^\top\bD^{-1}(z)\br_k
=
\frac1n\beta_k(z)\br_k^\top\bD_k^{-1}(z)\br_k
=
1-\beta_k(z),
\]
we have
\begin{align}\label{bet}
|\beta_k(z)|
\le
1+\frac Cn\br_k^\top\br_k
+C\eta_n^2n^{2}\va_n^{-1}
I\bigl(\|\bS_n\|\ge\eta_r\ {\rm or}\ \lambda_{\min}^{\bS_n}<\eta_l\bigr).
\end{align}
Hence, for $\ell\ge1$,
\[
\E|\beta_k(z)|^{\ell}\le C.
\]
Moreover, $|b(z)|$ is uniformly bounded for all $n$ and $z\in\mathcal C_n$.
Denote
\[
\zeta_k(z)=\br_k^\top\bD_k^{-1}(z)\br_k-\E\tr\bigl(\bD_1^{-1}(z)\bSig_n\bigr).
\]
Then it follows that
\begin{align}\label{bee}
b(z)=\beta_k(z)+\frac1n\beta_k(z)b(z)\zeta_k(z).
\end{align}
Using Lemma \ref{mlez} and Lemma \ref{burk}, we obtain, for $\ell\ge2$,
\begin{align}\label{meq2}
\begin{split}
\E|\zeta_k(z)|^{\ell}
\le&
\, C_{\ell}\eta_n^{\ell-2}n^{\ell-2}\bigl(n+n^{2\kappa}\bigr)
+C_{\ell}\E\left|
\tr\bigl(\bD_k^{-1}(z)\bSig_n\bigr)-\E\tr\bigl(\bD_k^{-1}(z)\bSig_n\bigr)
\right|^{\ell}\\
\le&
\, C_{\ell}\eta_n^{\ell-2}n^{\ell-2}\bigl(n+n^{2\kappa}\bigr)
+C_{\ell}\E\left|
\tr\bigl(\bD_k^{-1}(z)\bSig_n\bigr)-\tr\bigl(\bD^{-1}(z)\bSig_n\bigr)
\right|^{\ell}\\
&+\frac{C_{\ell}}{n^{\ell}}
\E\left(
\sum_{j=1 }^n
\left|
\beta_k(z)\br_k^\top\bD_k^{-1}(z)\bSig_n\bD_k^{-1}(z)\br_k
\right|^2
\right)^{\ell/2}\\
\le&
\, C_{\ell}\eta_n^{\ell-2}n^{\ell-2}\bigl(n+n^{2\kappa}\bigr).
\end{split}
\end{align}

By calculation, we have
\begin{align*}
\frac{M_{n1}(z_1)-M_{n1}(z_2)}{z_1-z_2}
=&\,
r_p(\kappa)\sum_{k=1}^n(\E_k-\E_{k-1})
\tr\bigl(\bD^{-1}(z_1)\bD^{-1}(z_2)\bigr)\\
=&\,
\frac{r_p(\kappa)}{n^2}\sum_{k=1}^n(\E_k-\E_{k-1})
\beta_k(z_1)\beta_k(z_2)
\bigl(\br_k^\top\bD_k^{-1}(z_1)\bD_k^{-1}(z_2)\br_k\bigr)^2\\
&-\frac{r_p(\kappa)}{n}\sum_{k=1}^n(\E_k-\E_{k-1})
\beta_k(z_1)\br_k^\top\bD_k^{-2}(z_1)\bD_k^{-1}(z_2)\br_k\\
&-\frac{r_p(\kappa)}{n}\sum_{k=1}^n(\E_k-\E_{k-1})
\beta_k(z_2)\br_k^\top\bD_k^{-1}(z_1)\bD_k^{-2}(z_2)\br_k\\
\triangleq&\,
\mathcal I_1+\mathcal I_2+\mathcal I_3.
\end{align*}
We shall show that
\[
\E|\mathcal I_j|^2\le C,\qquad j=1,2,3.
\]
Here we only prove the bound for $\mathcal I_1$; the other two terms are similar. By \eqref{bee}, we get
\begin{align*}
\mathcal I_1
=&\,
\frac{r_p(\kappa)}{n^2}b(z_1)b(z_2)\sum_{k=1}^n(\E_k-\E_{k-1})
\bigl(\br_k^\top\bD_k^{-1}(z_1)\bD_k^{-1}(z_2)\br_k\bigr)^2\\
&-\frac{r_p(\kappa)}{n^3}b(z_1)b(z_2)\sum_{k=1}^n(\E_k-\E_{k-1})
\beta_k(z_2)\zeta_k(z_2)
\bigl(\br_k^\top\bD_k^{-1}(z_1)\bD_k^{-1}(z_2)\br_k\bigr)^2\\
&-\frac{r_p(\kappa)}{n^3}b(z_1)\sum_{k=1}^n(\E_k-\E_{k-1})
\beta_k(z_1)\beta_k(z_2)\zeta_k(z_1)
\bigl(\br_k^\top\bD_k^{-1}(z_1)\bD_k^{-1}(z_2)\br_k\bigr)^2\\
\triangleq&\,
\mathcal I_{11}+\mathcal I_{12}+\mathcal I_{13}.
\end{align*}
It follows from Lemma \ref{mlez} that
\begin{align*}
\E|\mathcal I_{11}|^2
\le&
\frac{Cr_p^2(\kappa)}{n^4}\sum_{k=1}^n
\E\left|
\bigl(\br_k^\top\bD_k^{-1}(z_1)\bD_k^{-1}(z_2)\br_k\bigr)^2
-\bigl(\tr\bD_k^{-1}(z_1)\bD_k^{-1}(z_2)\bSig_n\bigr)^2
\right|^2\\
\le&
\frac{Cr_p^2(\kappa)}{n^4}\sum_{k=1}^n
\E\left|
\br_k^\top\bD_k^{-1}(z_1)\bD_k^{-1}(z_2)\br_k
-\tr\bD_k^{-1}(z_1)\bD_k^{-1}(z_2)\bSig_n
\right|^4\\
&+\frac{Cr_p^2(\kappa)}{n^2}\bigl(n+n^{2\kappa}\bigr)
\sum_{k=1}^n\E\|\bD_k^{-1}(z_1)\bD_k^{-1}(z_2)\|^4\\
\le&
\frac{Cr_p^2(\kappa)}{n}\bigl(n+n^{2\kappa}\bigr)\le C.
\end{align*}
By \eqref{bet} and Lemma \ref{mlezt}, one finds
\begin{align*}
\E|\mathcal I_{12}|^2
\le&
\frac{Cr_p^2(\kappa)}{n^6}\sum_{k=1}^n
\E\left|
\beta_k(z_2)\zeta_k(z_2)
\bigl(\br_k^\top\bD_k^{-1}(z_1)\bD_k^{-1}(z_2)\br_k\bigr)^2
\right|^2\\
\le&
\frac{Cr_p^2(\kappa)}{n^6}\sum_{k=1}^n
\E\left|
\zeta_k(z_2)
\bigl(\br_k^\top\bD_k^{-1}(z_1)\bD_k^{-1}(z_2)\br_k\bigr)^2
\right|^2\\
&+\frac{Cr_p^2(\kappa)}{n^8}\sum_{k=1}^n
\E\left|
\zeta_k(z_2)\br_k^\top\br_k
\bigl(\br_k^\top\bD_k^{-1}(z_1)\bD_k^{-1}(z_2)\br_k\bigr)^2
\right|^2\\
&+n^{21}\va_n^{-12}
{\rm P}\bigl(\|\bS_n\|\ge\eta_r\ {\rm or}\ \lambda_{\min}^{\bS_n}<\eta_l\bigr)
\le C,
\end{align*}
and similarly
\begin{align*}
\E|\mathcal I_{13}|^2\le C.
\end{align*}
Therefore, we conclude that $M_{n1}(z)$ is tight on $\mathcal C_n$.

\subsection{Convergence of $M_{n2}(z)$}

We shall show that for $z\in\mathcal C_n$, the sequence $\{M_{n2}(z)\}$
is bounded, forms a uniformly equicontinuous family, and converges to
\eqref{mvar}. It is straightforward to verify that
\begin{align}\label{meq5}
\sup_{n,\;z\in\mathcal{C}_n}\big|\E\underline m_n(z)-\underline m(z)\big|\to0
\end{align}
and
\begin{align*}
\sup_{n,\;z\in\mathcal{C}_n}
\left\|
\bigl(\E\underline m_n(z)\bSig_n+\bI_p\bigr)^{-1}
\right\|<\infty.
\end{align*}

Write $M_{n2}(z)$ as
\begin{align*}
M_{n2}(z)
=&\,
d_n(z)+pz^{-1}r_p(\kappa)\left[
-\int\frac1{1+\E\underline m_n(z)x}\,dH_n(x)
+\int\frac1{1+\underline m_n^0(z)x}\,dH_n(x)
\right]\\
=&\,
d_n(z)+pz^{-1}r_p(\kappa)\int
\frac{\bigl(\E \underline m_n(z)-\underline m_n^0(z)\bigr)x}
{\bigl(1+\E\underline m_n(z)x\bigr)\bigl(1+\underline m_n^0(z)x\bigr)}
\,dH_n(x),
\end{align*}
where
\[
d_n(z)=pr_p(\kappa)\left[
\E m_n(z)+\int\frac{z^{-1}}{1+\E\underline m_n(z)x}\,dH_n(x)
\right].
\]
Since
\[
M_{n2}(z)=pr_p(\kappa)\bigl[\E m_n(z)-m_n^0(z)\bigr]
=
nr_p(\kappa)\bigl[\E\underline m_n(z)-\underline m_n^0(z)\bigr],
\]
we obtain
\begin{align*}
M_{n2}(z)
=&\,
\frac{d_n(z)}
{1-c_nz^{-1}\int\frac{x}
{\bigl(1+\E\underline m_n(z)x\bigr)\bigl(1+\underline m_n^0(z)x\bigr)}\,dH_n(x)}\\
=&\,
-\frac{z\underline m_n^0(z)d_n(z)}
{1-c_n\int\frac{\E\underline m_n(z)\underline m_n^0(z)x^2}
{\bigl(1+\E\underline m_n(z)x\bigr)\bigl(1+\underline m_n^0(z)x\bigr)}\,dH_n(x)}.
\end{align*}
Thus, our goal is to find the limit of $d_n(z)$.

Let
\[
\bR(z)=-z\E\underline m_n(z)\bSig_n-z\bI_p.
\]
Then, by \eqref{mal7},
\[
\bR(z)=\E\beta_1(z)\bSig_n-z\bI_p.
\]
Write
\[
\bD(z)-\bR(z)
=
\frac1n\sum_{k=1}^n\br_k\br_k^\top-\E\beta_1(z)\bSig_n.
\]
Taking inverses and then expectations, we have
\begin{align}\label{sal:6}
\bR^{-1}(z)-\E\bD^{-1}(z)
=&\,
\bR^{-1}(z)\E\left[
\frac1n\sum_{k=1}^n\beta_k(z)\br_k\br_k^\top\bD_k^{-1}(z)
-\bigl(\E\beta_k(z)\bigr)\bSig_n\bD^{-1}(z)
\right].
\end{align}
Taking the trace on both sides and multiplying by $-r_p(\kappa)$, one obtains
\begin{align}\label{beq20}
d_n(z)
=&\,
-\frac{r_p(\kappa)}n\sum_{k=1}^n\E\beta_k(z)
\left(
\br_k^\top\bD_k^{-1}(z)\bR^{-1}(z)\br_k
-\E\tr\bigl(\bR^{-1}(z)\bSig_n\bD_k^{-1}(z)\bigr)
\right)\notag\\
&-\frac{r_p(\kappa)}n\sum_{k=1}^n\E\beta_k(z)
\left(
\E\tr\bigl(\bR^{-1}(z)\bSig_n\bD_k^{-1}(z)\bigr)
-\E\tr\bigl(\bR^{-1}(z)\bSig_n\bD^{-1}(z)\bigr)
\right)\notag\\
\triangleq&\,
\mathcal{J}_1+\mathcal{J}_2.
\end{align}

Using \eqref{bee}, it follows that
\begin{align*}
\mathcal{J}_1
=&\,
\frac{r_p(\kappa)}{n^2}\sum_{k=1}^n b^2(z)\E\zeta_k(z)
\left(
\br_k^\top\bD_k^{-1}(z)\bR^{-1}(z)\br_k
-\E\tr\bigl(\bR^{-1}(z)\bSig_n\bD_k^{-1}(z)\bigr)
\right)\\
&-\frac{r_p(\kappa)}{n^3}\sum_{k=1}^n b^2(z)\E\beta_k(z)\zeta_k^2(z)
\left(
\br_k^\top\bD_k^{-1}(z)\bR^{-1}(z)\br_k
-\E\tr\bigl(\bR^{-1}(z)\bSig_n\bD_k^{-1}(z)\bigr)
\right)\\
=&\,
\frac{r_p(\kappa)}{n^2}\sum_{k=1}^n b^2(z)\E\va_k(z)
\left(
\br_k^\top\bD_k^{-1}(z)\bR^{-1}(z)\br_k
-\tr\bigl(\bR^{-1}(z)\bSig_n\bD_k^{-1}(z)\bigr)
\right)\\
&+\frac{r_p(\kappa)}{n^2}\sum_{k=1}^n b^2(z)
{\rm Cov}\left(
\tr\bigl(\bD_k^{-1}(z)\bSig_n\bigr),
\tr\bigl(\bR^{-1}(z)\bSig_n\bD_k^{-1}(z)\bigr)
\right)\\
&-\frac{r_p(\kappa)}{n^3}\sum_{k=1}^n b^2(z)\E\beta_k(z)\zeta_k^2(z)
\left(
\br_k^\top\bD_k^{-1}(z)\bR^{-1}(z)\br_k
-\E\tr\bigl(\bR^{-1}(z)\bSig_n\bD_k^{-1}(z)\bigr)
\right)\\
\triangleq&\,
\mathcal{J}_{11}+\mathcal{J}_{12}+\mathcal{J}_{13}.
\end{align*}

Using the martingale difference property, \eqref{bet}, and Lemma \ref{mlez},
we have for any deterministic matrix $\bA$,
\begin{align*}
&\E\left|
\tr\bigl(\bD^{-1}(z)\bA\bigr)-\E\tr\bigl(\bD^{-1}(z)\bA\bigr)
\right|^2\\
=&\,
\frac1{n^2}\sum_{k=1}^n
\E\left|
(\E_k-\E_{k-1})\beta_k(z)\br_k^\top\bD_k^{-1}(z)\bA\bD_k^{-1}(z)\br_k
\right|^2\\
\le&\,
\frac C{n^2}\sum_{k=1}^n
\E\left|
\beta_k(z)\Bigl(
\br_k^\top\bD_k^{-1}(z)\bA\bD_k^{-1}(z)\br_k
-\tr\bigl(\bD_k^{-1}(z)\bA\bD_k^{-1}(z)\bSig_n\bigr)
\Bigr)
\right|^2\\
&+\frac C{n^2}\sum_{k=1}^n
\E\left|
\bigl(\beta_k(z)-\bar\beta_k(z)\bigr)
\tr\bigl(\bD_k^{-1}(z)\bA\bD_k^{-1}(z)\bSig_n\bigr)
\right|^2\\
\le&\,
C\bigl(1+n^{2\kappa-1}\bigr).
\end{align*}
By the same argument, we obtain
\begin{align}\label{meq1}
\E\left|
\tr\bigl(\bD_k^{-1}(z)\bA\bigr)-\E\tr\bigl(\bD_k^{-1}(z)\bA\bigr)
\right|^2
\le
C\bigl(1+n^{2\kappa-1}\bigr).
\end{align}
Due to \eqref{meq1} and the Cauchy--Schwarz inequality, one finds
\begin{align*}
|\mathcal{J}_{12}|
\le&
\frac {Cr_p(\kappa)}{n^2}\sum_{k=1}^n
{\rm Var}^{1/2}\left(\tr\bigl(\bD_k^{-1}(z)\bSig_n\bigr)\right)
{\rm Var}^{1/2}\left(\tr\bigl(\bR^{-1}(z)\bSig_n\bD_k^{-1}(z)\bigr)\right)
\le
\frac Cn.
\end{align*}
Applying \eqref{bet}, \eqref{meq1}, and Lemma \ref{mlezt}, we have
\begin{align*}
|\mathcal{J}_{13}|
\le&
\frac{r_p(\kappa)}{n^3}\sum_{k=1}^n
\E\left|
\beta_k(z)\zeta_k^2(z)
\left(
\br_k^\top\bD_k^{-1}(z)\bR^{-1}(z)\br_k
-\tr\bigl(\bR^{-1}(z)\bSig_n\bD_k^{-1}(z)\bigr)
\right)
\right|\\
&+\frac{r_p(\kappa)}{n^3}\sum_{k=1}^n
\E\left|
\beta_k(z)\zeta_k^2(z)
\left(
\tr\bigl(\bR^{-1}(z)\bSig_n\bD_k^{-1}(z)\bigr)
-\E\tr\bigl(\bR^{-1}(z)\bSig_n\bD_k^{-1}(z)\bigr)
\right)
\right|\\
\le&
\frac {Cr_p(\kappa)}{n^3}\bigl(\sqrt n+n^{\kappa}\bigr)
\sum_{k=1}^n
\E^{1/2}\left|\beta_k(z)\zeta_k^2(z)\right|^2
\le C\eta_n.
\end{align*}
Hence, we deduce
\begin{align}\label{cal4}
\mathcal{J}_1
=
\frac{r_p(\kappa)b^2(z)}{n^2}
\sum_{k=1}^n\E\va_k(z)
\left(
\br_k^\top\bD_k^{-1}(z)\bR^{-1}(z)\br_k
-\tr\bigl(\bR^{-1}(z)\bSig_n\bD_k^{-1}(z)\bigr)
\right)
+o(1).
\end{align}

By \eqref{bee}, it is obvious that
\begin{align*}
\mathcal{J}_2
=&\,
-\frac{r_p(\kappa)}{n^2}\sum_{k=1}^n
\E\beta_k(z)\E\left(
\beta_k(z)\br_k^\top\bD_k^{-1}(z)\bR^{-1}(z)\bSig_n\bD_k^{-1}(z)\br_k
\right)\\
=&\,
-\frac{r_p(\kappa)}{n^2}\sum_{k=1}^n
b^2(z)\E\tr\bigl(\bD_k^{-1}(z)\bR^{-1}(z)\bSig_n\bD_k^{-1}(z)\bSig_n\bigr)\\
&+\frac{r_p(\kappa)}{n^3}\sum_{k=1}^n
b^2(z)\E\left(
\beta_k(z)\zeta_k(z)\br_k^\top\bD_k^{-1}(z)\bR^{-1}(z)\bSig_n\bD_k^{-1}(z)\br_k
\right)\\
&+\frac{r_p(\kappa)}{n^3}\sum_{k=1}^n
b(z)\E\bigl(\beta_k(z)\zeta_k(z)\bigr)
\E\left(
\beta_k(z)\br_k^\top\bD_k^{-1}(z)\bR^{-1}(z)\bSig_n\bD_k^{-1}(z)\br_k
\right)\\
\triangleq&\,
\mathcal{J}_{21}+\mathcal{J}_{22}+\mathcal{J}_{23}.
\end{align*}
Using \eqref{meq2} and Lemma \ref{mlezt}, we get
\begin{align*}
|\mathcal{J}_{22}|
\le&
\frac {Cr_p(\kappa)}{n^3}\sum_{k=1}^n
\E^{1/2}\left|\zeta_k(z)\right|^2
\E^{1/4}\left|
\br_k^\top\bD_k^{-1}(z)\bR^{-1}(z)\bSig_n\bD_k^{-1}(z)\br_k
\right|^4
\le
\frac C{\sqrt n},
\end{align*}
and
\begin{align*}
|\mathcal{J}_{23}|
\le&
\frac {Cr_p(\kappa)}{n^3}\sum_{k=1}^n
\E^{1/2}\left|\zeta_k(z)\right|^2
\E^{1/2}\left|
\br_k^\top\bD_k^{-1}(z)\bR^{-1}(z)\bSig_n\bD_k^{-1}(z)\br_k
\right|^2
\le
\frac C{\sqrt n}.
\end{align*}
Combining the above two inequalities, it yields
\begin{align}\label{meq3}
\mathcal{J}_2
=
-\frac{r_p(\kappa)}{n^2}\sum_{k=1}^n
b^2(z)\E\tr\bigl(\bD_k^{-1}(z)\bR^{-1}(z)\bSig_n\bD_k^{-1}(z)\bSig_n\bigr)
+o(1).
\end{align}

Together with \eqref{beq20}, \eqref{cal4}, and \eqref{meq3}, we see
\begin{align*}
d_n(z)
=&\,
\frac{r_p(\kappa)b^2(z)}{n^2}\sum_{k=1}^n
\E\va_k(z)
\left(
\br_k^\top\bD_k^{-1}(z)\bR^{-1}(z)\br_k
-\tr\bigl(\bR^{-1}(z)\bSig_n\bD_k^{-1}(z)\bigr)
\right)\\
&-\frac{r_p(\kappa)}{n^2}\sum_{k=1}^n
b^2(z)\E\tr\bigl(\bD_k^{-1}(z)\bR^{-1}(z)\bSig_n\bD_k^{-1}(z)\bSig_n\bigr)
+o(1).
\end{align*}
By \eqref{mal7} and \eqref{meq2}, it is apparent that
\begin{align*}
|z\E\underline m_n(z)+b(z)|^2
\le&
\,\E|\beta_1(z)-b(z)|^2\\
\le&
\,\frac C{n^2}
\E\left|
\beta_k(z)\Bigl(
\br_k^\top\bD_k^{-1}(z)\br_k-\E\tr\bigl(\bD_k^{-1}(z)\bSig_n\bigr)
\Bigr)
\right|^2\\
\le&
\,\frac C{n^2}
\E^{1/2}\left|
\br_k^\top\bD_k^{-1}(z)\br_k-\E\tr\bigl(\bD_k^{-1}(z)\bSig_n\bigr)
\right|^4\\
\le&
\,\frac {C\eta_n(\sqrt n+n^{\kappa})}{n}\to0.
\end{align*}
Using Lemma \ref{mlez2}, \eqref{meq5}, and the above inequality, one finds
\begin{align*}
d_n(z)
=
&\,
\frac{\nu_{x4}r_p(\kappa)b^2(z)}{n^2}\sum_{k=1}^n
\tr\left[
\bigl(
\bga_n^\top\bD_k^{-1}(z)\bga_n
\circ
\bga_n^\top\bD_k^{-1}(z)\bR^{-1}(z)\bga_n
\bigr)\bF^1_{p_1}
\right]\\
&+\frac{\nu_{y4}r_p(\kappa)b^2(z)}{n^2}\sum_{k=1}^n
\tr\left[
\bigl(
\bga_n^\top\bD_k^{-1}(z)\bga_n
\circ
\bga_n^\top\bD_k^{-1}(z)\bR^{-1}(z)\bga_n
\bigr)\bF^2_{p_2}
\right]\\
&+\frac{\tau_xr_p(\kappa)b^2(z)}{p^{\delta_1}n^2}\sum_{k=1}^n
\tr\bigl(\bga_n^\top\bD_k^{-1}(z)\bga_n\bF_{p_1}^1\bigr)
\tr\bigl(\bga_n^\top\bD_k^{-1}(z)\bR^{-1}(z)\bga_n\bF_{p_1}^1\bigr)\\
&+\frac{\tau_yr_p(\kappa)b^2(z)}{p^{\delta_2}n^2}\sum_{k=1}^n
\tr\bigl(\bga_n^\top\bD_k^{-1}(z)\bga_n\bF_{p_2}^2\bigr)
\tr\bigl(\bga_n^\top\bD_k^{-1}(z)\bR^{-1}(z)\bga_n\bF_{p_2}^2\bigr)\\
&+\frac{r_p(\kappa)}{n^2}\sum_{k=1}^n
b^2(z)\E\tr\bigl(\bD_k^{-1}(z)\bR^{-1}(z)\bSig_n\bD_k^{-1}(z)\bSig_n\bigr)
+o(1)\\
=&\,
-\frac{\nu_{x4}r_p(\kappa)\underline m^2(z)}{zn}
\tr\left[
\bigl(
\bga_n^\top(\underline m(z)\bSig_n+\bI_p)^{-1}\bga_n
\circ
\bga_n^\top(\underline m(z)\bSig_n+\bI_p)^{-2}\bga_n
\bigr)\bF^1_{p_1}
\right]\\
&-\frac{\nu_{y4}r_p(\kappa)\underline m^2(z)}{zn}
\tr\left[
\bigl(
\bga_n^\top(\underline m(z)\bSig_n+\bI_p)^{-1}\bga_n
\circ
\bga_n^\top(\underline m(z)\bSig_n+\bI_p)^{-2}\bga_n
\bigr)\bF^2_{p_2}
\right]\\
&-\frac{\tau_xr_p(\kappa)\underline m^2(z)}{zp^{\delta_1}n}
\tr\bigl(\bga_n^\top(\underline m(z)\bSig_n+\bI_p)^{-1}\bga_n\bF_{p_1}^1\bigr)
\tr\bigl(\bga_n^\top(\underline m(z)\bSig_n+\bI_p)^{-2}\bga_n\bF_{p_1}^1\bigr)\\
&-\frac{\tau_yr_p(\kappa)\underline m^2(z)}{zp^{\delta_2}n}
\tr\bigl(\bga_n^\top(\underline m(z)\bSig_n+\bI_p)^{-1}\bga_n\bF_{p_2}^2\bigr)
\tr\bigl(\bga_n^\top(\underline m(z)\bSig_n+\bI_p)^{-2}\bga_n\bF_{p_2}^2\bigr)\\
&-\frac{r_p(\kappa)\underline m^2(z)}{zn}
\frac{\tr\bigl((\underline m(z)\bSig_n+\bI_p)^{-3}\bSig_n^2\bigr)}
{1-\frac{\underline m^2(z)}{n}\tr\bigl((\underline m(z)\bSig_n+\bI_p)^{-2}\bSig_n^2\bigr)}
+o(1).
\end{align*}
Therefore, we deduce
\begin{align*}
M_{n2}(z)
=&\,
\frac{
cr_p(\kappa)\int\frac{\underline m^3(z)x^2}
{\left(1+\underline m(z)x\right)^3}\,dH(x)}
{\left(
1-c\int\frac{\underline m^2(z)x^2}
{\left(1+\underline m(z)x\right)^2}\,dH(x)
\right)^2}\\
&+\frac{
r_p(\kappa)\underline m^3(z)}
{1-c\int\frac{\underline m^2(z)x^2}
{\left(1+\underline m(z)x\right)^2}\,dH(x)}\\
&\times\Bigg\{
\frac{\nu_{x4}}{n}
\tr\left[
\bigl(
\bga_n^\top\bH^{-1}(z)\bga_n
\circ
\bga_n^\top\bH^{-2}(z)\bga_n
\bigr)\bF^1_{p_1}
\right]\\
&\qquad
+\frac{\nu_{y4}}{n}
\tr\left[
\bigl(
\bga_n^\top\bH^{-1}(z)\bga_n
\circ
\bga_n^\top\bH^{-2}(z)\bga_n
\bigr)\bF^2_{p_2}
\right]\\
&\qquad
+\frac{\tau_x}{p^{\delta_1}n}
\tr\bigl(\bga_n^\top\bH^{-1}(z)\bga_n\bF_{p_1}^1\bigr)
\tr\bigl(\bga_n^\top\bH^{-2}(z)\bga_n\bF_{p_1}^1\bigr)\\
&\qquad
+\frac{\tau_y}{p^{\delta_2}n}
\tr\bigl(\bga_n^\top\bH^{-1}(z)\bga_n\bF_{p_2}^2\bigr)
\tr\bigl(\bga_n^\top\bH^{-2}(z)\bga_n\bF_{p_2}^2\bigr)
\Bigg\}
+o(1),
\end{align*}
which coincides with the mean function stated in \eqref{mvar}. This proves the convergence of $M_{n2}(z)$.

\subsection{Completion of the proof of Lemma \ref{th2} and Theorem \ref{thm2}}

We now combine the fluctuation part $M_{n1}(z)$ and the deterministic centering
part $M_{n2}(z)$.

From the previous subsections, for $z\in\mathcal C_n$, we have
\[
M_n(z)=M_{n1}(z)+M_{n2}(z),
\]
where $M_{n1}(z)$ converges weakly to a centered Gaussian process, and
$M_{n2}(z)$ converges uniformly in probability on $\mathcal C_n$ to the
deterministic function given in \eqref{mvar}. Hence $\widehat M_n(z)$ converges
weakly on $\mathcal C$ to a Gaussian process $M(z)$ with mean function
\eqref{mvar} and covariance kernel
\begin{align*}
{\rm Cov}\bigl(M(z_1),M(z_2)\bigr)
=&\ 
2r_p^2(\kappa)\Bigg(
\frac{\underline m'(z_1)\underline m'(z_2)}
{\bigl(\underline m(z_1)-\underline m(z_2)\bigr)^2}
-\frac1{(z_1-z_2)^2}
\Bigg)
\\
&\ +\underline m'(z_1)\underline m'(z_2)
\lim_{n\to\infty}\frac{r_p^2(\kappa)}{n}
\Gamma_n^{\mathrm{blk}}
\Big(
(\underline m(z_1)\bSig_n+\bI_p)^{-1},
(\underline m(z_2)\bSig_n+\bI_p)^{-1}
\Big).
\end{align*}
Writing out $\Gamma_n^{\mathrm{blk}}$ explicitly gives exactly the covariance
formula stated in Lemma \ref{th2}.

It remains to pass from the process $\widehat M_n$ on the contour to the linear
spectral statistics. By Cauchy's integral formula,
\[
\int f(x)\,dG_n(x)
=
-\frac1{2\pi i}\oint_{\mathcal C\cup\overline{\mathcal C}} f(z)M_n(z)\,dz .
\]
Since $\widehat M_n(z)=M_n(z)$ on $\mathcal C_n$ and is frozen only on the two
removed short vertical segments, while the contribution of these segments is
negligible, we also have
\[
\int f(x)\,dG_n(x)
=
-\frac1{2\pi i}\oint_{\mathcal C\cup\overline{\mathcal C}} f(z)\widehat M_n(z)\,dz
+o_p(1).
\]
Now the mapping
\[
h(\cdot)\mapsto
\left(
-\frac1{2\pi i}\oint f_1(z)h(z)\,dz,\dots,
-\frac1{2\pi i}\oint f_s(z)h(z)\,dz
\right)
\]
is continuous on the space of continuous functions on
$\mathcal C\cup\overline{\mathcal C}$. Therefore, by the continuous mapping
theorem,
\[
\left(\int f_1(x)\,dG_n(x),\dots,\int f_s(x)\,dG_n(x)\right)
\xrightarrow{d}
(X_{f_1},\dots,X_{f_s}),
\]
where $(X_{f_1},\dots,X_{f_s})$ is Gaussian with mean and covariance given by
the contour integrals of the mean and covariance kernel of $M(\cdot)$. This
completes the proof of Theorem \ref{thm2}.

\section{Technical lemmas}

In this section we collect several technical lemmas used repeatedly in the proof
of Theorem~\ref{thm2}.

\begin{lemma}\label{mlez}
Under the conditions of Theorem \ref{thm2} and \eqref{mal2t}, we have for a non-random matrix $\bA$ and $\ell\ge2$,
\begin{align*}
	\E\left|\bw_1^\top\bA\bw_1-\tr\(\bA\)\right|^{\ell}\le		C_{\ell}\eta_n^{\ell-2}n^{\ell-2}\(n+n^{2\kappa}\)\|\bA\|^{\ell}.
\end{align*}
\end{lemma}
\begin{proof}
By Lemma \ref{mle3}, we get
\begin{align*}
	&\E\left|\bw_1^\top\bA\bw_1-\tr\(\bA\)\right|^{\ell}\\
\le&C_{\ell}\E\bigg(\(\left\|\bA\right\|^2\tr
\(\rho_{x1}^4\bF^1_{p_1}+\rho_{y1}^4\bF^2_{p_2}\)
\)^{\ell/2}+\eta_n^{2\ell-4}n^{\ell-2}\left\|\bA\right\|^{\ell}\tr\begin{pmatrix}
\rho_{x1}^{2\ell}\bF^1_{p_1}+&\rho_{y1}^{2\ell}\bF^2_{p_2}
\end{pmatrix}\bigg)\\
&+C_{\ell}\E\left|(\rho_{x1}^2-1)\tr\(\bA\bF^1_{p_1}\)+(\rho_{y1}^2-1)\tr\(\bA\bF^2_{p_2}\)\right|^{\ell}\\
\le&C_{\ell}\eta_n^{2\ell-4}n^{\ell-2}\|\bA\|^{\ell}\(p_1\E\rho_{x1}^{2\ell}+p_2\E\rho_{y1}^{2\ell}\)+C_{\ell}\|\bA\|^{\ell}\(p_1^{\ell}\E\left|\rho_{x1}^2-1\right|^{\ell}+p_2^{\ell}\E\left|\rho_{y1}^2-1\right|^{\ell}\)\\
\le&
	C_{\ell}\eta_n^{\ell-2}n^{\ell-2}\(n+n^{2\kappa}\)\|\bA\|^{\ell}.
\end{align*}
\end{proof}

\begin{lemma}\label{mlezt}
	Assume the conditions of Theorem \ref{thm2} and \eqref{mal2t} hold. For a non-random matrix $\bA_k,k=1,\cdots,s\ge0,$ and $\bB_{\ell},\ell=1,\cdots,t\ge2$, we have 
	\begin{align*}
	\E\left|\prod_{k=1}^s\bw_1^\top\bA_k\bw_1\prod_{\ell=1}^t\(\bw_1^\top\bB_{\ell}\bw_1-\tr\(\bB_{\ell}\)\)\right|\le		C_{\ell}\eta_n^{t-2}n^{t+s-2}\(n+n^{2\kappa}\)\prod_{k=1}^s\|\bA_k\|\prod_{\ell=1}^{t}\|\bB_{\ell}\|.
	\end{align*}
\end{lemma}
\begin{proof}
	When $s=0$ and $t\ge2$, we can obtain from Lemma \ref{mlez}
	\begin{align*}
\E\bigg|\prod_{\ell=1}^t\Big(\bw_1^\top\bB_{\ell}\bw_1-&\tr\(\bB_{\ell}\)\Big)\bigg|
\le\prod_{\ell=1}^t\re^{1/t}\left|\bw_1^\top\bB_{\ell}\bw_1-\tr\(\bB_{\ell}\)\right|^t\\
		\le&
	C_{\ell}\eta_n^{t-2}n^{t-2}\(n+n^{2\kappa}\)\prod_{\ell=1}^{t}\|\bB_{\ell}\|.
	\end{align*}
If $s\ge1$, then by induction on $s$, it follows that
\begin{align*}
	&\E\left|\prod_{k=1}^s\bw_1^\top\bA_k\bw_1\prod_{\ell=1}^t\(\bw_1^\top\bB_{\ell}\bw_1-\tr\(\bB_{\ell}\)\)\right|\\
\le	&\E\left|\prod_{k=1}^{s-1}\bw_1^\top\bA_k\bw_1\(\bw_1^\top\bA_s\bw_1-\tr\(\bA_s\)\)\prod_{\ell=1}^t\(\bw_1^\top\bB_{\ell}\bw_1-\tr\(\bB_{\ell}\)\)\right|\\
&+p\|\bA_s\|	\E\left|\prod_{k=1}^{s-1}\bw_1^\top\bA_k\bw_1\prod_{\ell=1}^t\(\bw_1^\top\bB_{\ell}\bw_1-\tr\(\bB_{\ell}\)\)\right|\\
\le&C_{\ell}\eta_n^{t-2}n^{t+s-2}\(n+n^{2\kappa}\)\prod_{k=1}^s\|\bA_k\|\prod_{\ell=1}^{t}\|\bB_{\ell}\|.
\end{align*} 
\end{proof}

\begin{lemma}\label{mlez2}
	Under the conditions of Theorem \ref{thm2}, we have for a square matrix $\bA$ and a symmetric matrix $\bB$,
	\begin{align*}
		\E\bigg[&\left(\bw_1^\top\bA\bw_1-\tr\(\bA\)\right)\left(\bw_1^\top\bB\bw_1-\tr\(\bB\)\right)\bigg]=\bigg[\nu_{x4}\tr\left[\(\bA\circ\bB\)\bF^1_{p_1}\right]+\nu_{y4}\tr\left[\(\bA\circ\bB\)\bF^2_{p_2}\right]\\
		&+2\tr\(\bA\bB\)+\tau_xp^{-\delta_1}\tr\(\bA\bF_{p_1}^1\)\tr\(\bB\bF_{p_1}^1\)+\tau_yp^{-\delta_2}\tr\(\bA\bF_{p_2}^2\)\tr\(\bB\bF_{p_2}^2\)\bigg](1+o(1)).
	\end{align*}
\end{lemma}
\begin{proof}
	By calculating, it yields
	\begin{align*}
		&\E\bigg[\left(\bw_1^\top\bA\bw_1-\tr\(\bA\)\right)\left(\bw_1^\top\bB\bw_1-\tr\(\bB\)\right)\bigg]\\
=&\E\bigg[\left(\bw_1^\top\bA\bw_1-\tr\(\bA
	\(\rho_{x1}^2\bF_{p_1}^1+\rho_{y1}^2\bF_{p_2}^2\)
\)\right)\left(\bw_1^\top\bB\bw_1-\tr\(\bB\(
\rho_{x1}^2\bF_{p_1}^1+\rho_{y1}^2\bF_{p_2}^2\)
\)\right)\bigg]\\
&+\E\tr\(\bA
	(\rho_{x1}^2-1)\bF^1_{p_1}+(\rho_{y1}^2-1)\bF^2_{p_2}
\)\tr\(\bB
	(\rho_{x1}^2-1)\bF^1_{p_1}+(\rho_{y1}^2-1)\bF^2_{p_2}
\)\\
=&\bigg[2\tr\(\bA\bB\)+\nu_{x4}\tr\left[\(\bA\circ\bB\)\bF^1_{p_1}\right]+\nu_{y4}\tr\left[\(\bA\circ\bB\)\bF^2_{p_2}\right]\\
&+\tau_xp^{-\delta_1}\tr\(\bA\bF_{p_1}^1\)\tr\(\bB\bF_{p_1}^1\)+\tau_yp^{-\delta_2}\tr\(\bA\bF_{p_2}^2\)\tr\(\bB\bF_{p_2}^2\)\bigg](1+o(1)).
	\end{align*}
\end{proof}

\begin{lemma}[Burkholder inequality]\label{burk}
Let $\{X_k\}$ be a complex martingale difference sequence with respect to the increasing $\sigma$-field $\{\mathcal F_k\}$. Then, we have for $\ell>1$,
\begin{align*}
	\E\left|\sum X_k\right|^{\ell}\le C_{\ell}\E\(\sum|X_k|^2\)^{\ell/2}.
\end{align*}
\end{lemma}

\begin{lemma}[Lemma 2.6 in \cite{Silverstein95S}]\label{mle2}
Let $z\in\mathbb{C}^+$ with $\Im(z)$, $\bA$ and $\bB$ $n\times n$ with $\bB$ Hermitian, and $\br\in\mathbb{C}^n$. Then
\begin{align*}
	\left|\tr\left[\(\bB-z\bI_n\)^{-1}\bA\right]-\tr\left[\(\bB+\br\br^\top-z\bI_n\)^{-1}\bA\right]\right|\le \frac{\|\bA\|}v.
\end{align*}
\end{lemma}

\begin{lemma}\label{mle3}
Let $\bA=(a_{jk})$ be an $n\times n$ random matrix and the entries $x_k$ of $\bx$ be independent random variables. Assume $\E\(x_k\)=0$, $\E\(x_k^2\)=1$, and $\E\(x_k\)^{\ell}\le \nu_{\ell}$. Then for $r\ge1$,
	\begin{align*}
		\E\left|\bx^\top\bA\bx-\tr\bA\right|^r\le C_{r}\bigg(\(\nu_4\tr\bA\bA^\top\)^{r/2}+\nu_{2r}\tr\(\bA\bA^\top\)^{r/2}\bigg).
	\end{align*}
\end{lemma}

\begin{lemma}\label{mleclt}
Suppose that for each $n$, $Y_{n1},\cdots,Y_{nr_n}$ is a  martingale difference sequence with respect to the increasing $\sigma$-field $\{\mathcal F_{nj}\}$ having second moments. If as $n\to\infty$,
\begin{itemize}
	\item [(i)] $$\sum_{j=1}^{r_n}\E\(Y_{nj}^2|\mathcal F_{n,j-1}\)\xrightarrow{i.p.}\sigma^2,\ (\sigma^2\ {\rm is} \ {\text a\ positive\ constant}),$$
	\item[(ii)] for each $\va>0$, $$\sum_{j=1}^{r_n}\E\(Y_{nj}^2I\(|Y_{nj}|\ge\va\)\)\to0,$$
\end{itemize}
then
$$\sum_{j=1}^{r_n}Y_{nj}\xrightarrow{d}N(0,\sigma^2).$$
\end{lemma}

\section{Additional derivation for the corrected John's statistic}\label{supp:john}

In this section, we provide the calculation omitted from the main text and prove Proposition~\ref{prop:john_bias_terms}. Throughout this section, we work under the spherical null
\[
H_0:\quad \bSig_n=\bI_p.
\]

Recall from \eqref{eq:M1n_main} that the model dependent mean correction is
\begin{align}
\mathcal M_{1,n}(f)
&=
-\frac{\gamma_n}{2\pi i\,n}
\oint_{\mathcal C}
f(z)\,
\frac{
\big(\underline m_n^0(z)\big)^3
}{
(1+\underline m_n^0(z))^3
\Bigl(
1-c_n\frac{(\underline m_n^0(z))^2}{(1+\underline m_n^0(z))^2}
\Bigr)
}
\,dz.
\label{supp:eq:M1_spherical_z}
\end{align}

We now introduce the change of variable
\[
u=\underline m_n^0(z).
\]
Under the spherical null, the companion Mar\v{c}enko--Pastur equation yields
\[
z=-\frac1u+\frac{c_n}{1+u},
\qquad
\frac{dz}{du}
=
\frac1{u^2}-\frac{c_n}{(1+u)^2}.
\]
Using this identity, we compute
\begin{align*}
\frac{u^3}{(1+u)^3}
\cdot
\frac{1}{
1-c_n\frac{u^2}{(1+u)^2}
}\,dz
&=
\frac{u^3}{(1+u)^3}
\cdot
\frac{1}{
1-c_n\frac{u^2}{(1+u)^2}
}
\left(
\frac1{u^2}-\frac{c_n}{(1+u)^2}
\right)\,du\\
&=
\frac{u^3}{(1+u)^3}
\cdot
\frac{1}{
1-c_n\frac{u^2}{(1+u)^2}
}
\cdot
\frac{(1+u)^2-c_nu^2}{u^2(1+u)^2}\,du\\
&=
\frac{u^3}{(1+u)^3}
\cdot
\frac{1}{
1-c_n\frac{u^2}{(1+u)^2}
}
\cdot
\frac{1-c_nu^2/(1+u)^2}{u^2}\,du\\
&=
\frac{u}{(1+u)^3}\,du.
\end{align*}
Hence \eqref{supp:eq:M1_spherical_z} reduces to
\begin{align}
\mathcal M_{1,n}(f)
=
-\frac{\gamma_n}{2\pi i\,n}
\oint_{\Gamma_u}
f\bigl(z(u)\bigr)\frac{u}{(1+u)^3}\,du,
\label{supp:eq:M1_spherical_u}
\end{align}
where $\Gamma_u$ is the image contour in the $u$-plane.

We now evaluate \eqref{supp:eq:M1_spherical_u} for the two test functions relevant to John's statistic.

\subsection*{The case $f_1(x)=x$}

Using
\[
z(u)=-\frac1u+\frac{c_n}{1+u},
\]
we obtain
\[
f_1\bigl(z(u)\bigr)\frac{u}{(1+u)^3}
=
\left(-\frac1u+\frac{c_n}{1+u}\right)\frac{u}{(1+u)^3}
=
-\frac1{(1+u)^3}
+\frac{c_nu}{(1+u)^4}.
\]
This function is analytic at $u=0$, so its residue at $u=0$ is zero. Therefore
\[
\mathcal M_{1,n}(f_1)=0.
\]

\subsection*{The case $f_2(x)=x^2$}

For $f_2(x)=x^2$, we have
\begin{align*}
f_2\bigl(z(u)\bigr)\frac{u}{(1+u)^3}
&=
\left(-\frac1u+\frac{c_n}{1+u}\right)^2\frac{u}{(1+u)^3}\\
&=
\left(
\frac1{u^2}
-\frac{2c_n}{u(1+u)}
+\frac{c_n^2}{(1+u)^2}
\right)\frac{u}{(1+u)^3}\\
&=
\frac1{u(1+u)^3}
-\frac{2c_n}{(1+u)^4}
+\frac{c_n^2u}{(1+u)^5}.
\end{align*}
Only the first term has a pole at $u=0$, and its residue equals $1$. Hence
\[
\mathcal M_{1,n}(f_2)
=
-\frac{\gamma_n}{n}.
\]

This completes the proof of Proposition~\ref{prop:john_bias_terms}.

\bibliographystyle{imsart-nameyear}

\end{document}